\documentclass[11pt,letterpaper]{amsart}
\usepackage{amssymb,amsmath, amsthm,latexsym}
\usepackage[T1]{fontenc}
\usepackage[utf8]{inputenc}
\usepackage{graphics}
\usepackage{todonotes}
\usepackage{amscd}
\usepackage{tikz-cd}
\usepackage{graphics}
\usepackage{here}
\newcommand{\cal}[1]{\mathcal{#1}}
\theoremstyle{plain}
\newtheorem*{theo}{Theorem}
\newtheorem{lemma}{Lemma}[section]
\newtheorem{theorem}[lemma]{Theorem}
\newtheorem{proposition}[lemma]{Proposition}
\newtheorem{corollary}[lemma]{Corollary}
 
\theoremstyle{definition}

\newtheorem{remark}[lemma]{Remark}
\parskip=\bigskipamount
\newcommand{\Ric}{\mathrm{Ric}\,}
\newcommand{\omr}{\omega_R}
\newcommand{\womr}{\widehat \omega_R}

\newcommand{\KN}{\mathbin{\bigcirc\mspace{-15mu}\wedge\mspace{3mu}}}

\let\phialt=\phi
\let\phi=\varphi
\let\varphi=\phialt
\begin{document}
\title{K\"ahler--Einstein metrics of negative curvature}
%\author{Ursula Hamenst\"adt}
\thanks{AMS subject classification: 32Q20, 53C55, 32Q02}
\date{April 22, 2025}

\author{Henri Guenancia}
\address{Univ. Bordeaux, CNRS, Bordeaux INP, IMB, UMR 5251, F-33400 Talence, France}
\email{henri.guenancia@math.cnrs.fr}
\author{Ursula Hamenstädt}
\address{Mathematisches Institut der Universität Bonn, Endenicher Allee 60, 53115 Bonn, Germany}
\email{ursula@math.uni-bonn.de}
\begin{abstract}
    Given any integer $n\ge 2$, we construct a compact Kähler--Einstein manifold of dimension $n$ of negative 
    sectional curvature which is not covered by the ball. 
\end{abstract}

\maketitle

\section{Introduction}

An important problem in complex geometry consists in finding   compact complex manifolds $M$ admitting a hermitian metric $\omega$ with good curvature properties. Formulated as such, the problem is of course vague and there are many ways to make it more precise.  In what follows, we will be exclusively interested in Kähler metrics, that is, we will impose that $d\omega=0$. 

Given a compact Kähler manifold $(M,\omega)$, there exist several distinct notions of curvature, e.g. the sectional curvature ($K_\omega$), the holomorphic bisectional curvature ($\mathrm{HBC}_\omega$), the holomorphic sectional curvature ($\mathrm{HSC}_\omega$), the Ricci curvature ($\Ric_\omega$) and the scalar curvature ($s_\omega$). Although each of these objects are tensors of different types, it makes sense to talk about (semi)posivity or (semi)negativity of these curvatures. Then we have the following implications 
\[\begin{tikzcd}
    K_\omega <0 &\Longrightarrow &\mathrm{HBC}_\omega<0 \arrow[d, Rightarrow]&\Longrightarrow &\mathrm{HSC}_\omega<0  \arrow[d, Rightarrow]\\
    &&\Ric_\omega <0 & \Longrightarrow &s_\omega <0
\end{tikzcd}\]
\noindent
and similarly with seminegativity or (semi)positivity. 
%An even stronger notion of negative curvature exists; it was exhibited by Siu \cite{Siu} and amounts to asking the holomorphic cotangent bundle 
%$(\Omega_M, \omega)$ to be Nakano positive. This positivity forces $M$ to be holomorphically rigid. 

%The \marginpar{H. I'm not sure if we want to leave this paragraph} 
%\marginpar{U. I suggest not to keep it, it distracts a bit}
%classical uniformization theorem (see e.g. \cite{Tian}) states that if $(M,\omega)$ is a compact Kähler manifold with \emph{constant} holomorphic bisectional curvature, then up to scale the universal cover $(\widetilde M, \widetilde \omega)$ is isomorphic to one of the three model spaces $( \mathbb B^n, \omega_B)$, $(\mathbb C^n, \omega_{\rm eucl})$ or $(\mathbb P^n, \omega_{\rm FS})$ depending on whether the sign of the curvature is negative, zero or positive. Here, $ \mathbb B^n= \{z\in \mathbb C^n; |z|^2<1\}$ is the unit ball, $\omega_B:=-dd^c \log(1-|z|^2)$ is the Bergman metric on $\mathbb B^n$, $\omega_{\rm eucl}= dd^c |z|^2$ and $\omega_{\rm FS}$ is the Fubini-Study metric given in a standard chart $\mathbb C^n\subset \mathbb P^n$ by the formula $\omega_{\rm FS}=dd^c \log(1+|z|^2)$. 

If $(M,\omega)$ is a compact Kähler manifold with positive bisectional curvature, a celebrated theorem of Siu and Yau \cite{SY} implies that $M$ is biholomorphic to the projective space, cf also Mori's theorem \cite{Mori} in the algebraic setting. In the negative curvature case, that is, if $(M,\omega)$ has negative sectional curvature, it was asked by Yau in \cite{Yau82} whether the universal cover $\widetilde M$ is biholomorphic to the ball $\mathbb B^n$. It turns out that this question has a negative answer. Throughout the years several counterexamples have been exhibited, e.g. in dimension two by Mostow and Siu \cite{MS}, in dimension three by Deraux \cite{De}, and in any dimension by Stover and Toledo \cite{ST22}.

The examples of Mostow-Siu, of Deraux and of Stover-Toledo have infinite fundamental group. They are either finite branched covers of ball quotients 
(the examples of Mostow-Siu and of Stover-Toledo), 
or their universal covers can locally be described as branched covers of the ball (the examples of Deraux). 
That the manifolds found by Stover-Toledo admit K\"ahler metrics with negative
definite complex curvature operator and hence negative curvature in the sense of Siu \cite{Siu}
%and hence are holomorphically rigid 
follows from an earlier result of 
Zheng \cite{Zh96}. As consequence \cite{Siu}, the manifolds are \emph{holomorphically rigid}: 
any compact complex manifold which is homotopy equivalent to one of these manifolds is biholomorphic to it. Minemyer \cite{M25} equipped these manifolds, called Stover-Toledo manifolds in the sequel,
with non-K\"ahler Riemannian 
metrics whose Riemannian curvature operator is non-positive. 

In a different direction, Mohsen \cite{Moh} produced examples of {\it simply connected} Kähler manifolds $(M,\omega)$ such that the holomorphic bisectional curvature is negative. These Kähler manifolds are complete intersections (of large codimension) in the projective space endowed with the restriction of the Fubini-Study metric. %showed that any compact manifold $M$ which is a finite cover of a compact ball quotient 
%$M\to \Gamma\backslash \mathbb B^n$, 
%ramified over a totally geodesic embedded hypersurface $D\subset \Gamma\backslash \mathbb B^n$,
%
%However, it was discovered 
%only very recently by Stover and Toledo \cite{ST22} that  
%such branched coverings do exist in any dimension $n\geq 2$, and additionally, 
%they are not quotients of the ball. Subsequently, 

%\marginpar{H. Should be mention Hummel-Schröder here too (and maybe the recent results of Sarem)?}\marginpar{U. Perhaps not, the introduction gets %a bit long}

%Of course, once such a $(M,\omega)$ is produced, one gets many more simply by perturbing the initial Kähler metric. 
This leaves open the question whether there are
"canonical" Kähler metrics of negative curvature on compact complex manifolds 
which are not locally symmetric. 
More precisely, we ask about the existence of a non-locally symmetric compact Kähler manifold $(M,\omega)$ such that
\[K_\omega <0 \quad \mbox{and} \quad \Ric \omega= c\, \omega \]
where $c\in \mathbb R$ is a (negative) so-called Einstein constant. 

Thanks to a celebrated theorem of Aubin \cite{Au} and Yau \cite{Yau78}, it is known that a compact Kähler manifold $M$ admits a unique normalized Kähler--Einstein metric of negative Ricci curvature, that is, 
a Kähler metric $\omega$ such that $\Ric(\omega)=-\omega$, if and only if  the first Chern class of $M$ is negative, in the sense that there exists a Kähler metric in the class $-c_1(M)$. If this  cohomological condition is satisfied, the unique Kähler--Einstein metric is constructed indirectly by solving a complex Monge-Ampère equation. However, as it is in general impossible to read 
off from the latter
partial differential equation information on the sectional curvature, 
%properties the (Chern or Riemann) curvature tensor 
the above question is quite delicate. Our main result is the following. 

\begin{theo}
For every $n\geq 2$ there exists a compact complex manifold $M$ of dimension $n$ 
whose universal covering is not biholomorphic to the ball and
which admits a Kähler--Einstein metric of negative sectional curvature.
\end{theo}

Actually one can obtain the following refined statement. For an a priori chosen constant 
$\epsilon >0$ and any number $n\geq 2$, 
there exists a compact Kähler--Einstein manifold $(M_\varepsilon, g_\varepsilon)$ of dimension $n$ 
and Einstein constant $-1$ 
such that the sectional curvature $\kappa$ of $g_\varepsilon$ satisfies
\[\min \kappa \in [-1,-1+\epsilon] \quad \mbox{ and} \quad  
\max \kappa\in [-\epsilon,0).\] 
In particular, we can find in any given complex dimension $n$ an infinite countable family of Kähler--Einstein manifolds $(M_k,g_k)_{k\in \mathbb N}$ of negative curvature whose universal covers $\widetilde M_{k}$ are mutually non biholomorphic.
All of these examples are Stover-Toledo manifolds. Furthermore, 
the Kähler--Einstein metrics in the theorem have very strongly negative curvature tensor 
in the sense of Siu. We refer to the last paragraph of the article for more information.

An old conjecture (see p.322 of \cite{MS} for an explicit statement) predicts that a simply connected
complete K\"ahler manifold of negative sectional curvature is biholomorphic to a bounded domain in 
$\mathbb{C}^n$. This is open even if one requires the metric to be K\"ahler-Einstein. We believe that our examples
hint at the possibility that this conjecture is not true. 

\noindent
{\bf Relation to earlier work.}
The question on the existence of negatively curved Einstein metrics on closed manifolds
which do not admit a locally symmetric metric also makes sense in the non-complex setting. The first examples of such 
metrics are due to Fine and Premoselli \cite{FP20}. They considered 
suitably chosen branched covers of some real hyperbolic four-manifolds (which 
in contrast to the complex setting are fairly easy to construct) and 
were able to show that an explicit negatively curved approximate Einstein metric on the branched cover can be perturbed to 
a negatively curved Einstein metric. 
This construction was extended in \cite{HJ24} to any dimension at least four. 
In the real setting, it is the \emph{existence} of some Einstein metric on a given closed 
negatively curved manifold of (real) dimension at least four which is difficult to establish. 
In the complex world, the 
existence of a K\"ahler Einstein metric can be read off from complex analytic invariants 
of the manifold, and the interest lies in the relation between 
geometric properties of a K\"ahler (Einstein) metric and complex  analytic properties of the manifold. 
%The approach we pursue for the proof of the main Theorem is inspired by \cite{FP20} as well. 

\noindent
{\bf Strategy of proof.}
Let $M:=\Gamma\backslash B$ be a compact quotient of the unit ball $B\subset \mathbb C^n$ by a torsion free uniform arithmetic lattice of simple type 
admitting a totally geodesic embedded smooth complex hypersurface $D\subset M$. Such lattices $\Gamma\subset \mathrm{PU}(n,1)$ 
are the starting point for the work of Stover and Toledo (see \cite{ST22}). 
We fix an integer $d\ge 2$. 

\emph{Step 1. Produce an orbifold model Kähler--Einstein metric $\omega_d$ near $D$.}

\noindent
Let $B_0\subset B$ be the totally geodesic complex hypersurface $B_0:=\{z_1=0\}\cap B$. Thanks to 
the theorem of Cheng-Yau \cite{CY80}, there exists on $B$ a unique complete Kähler--Einstein metric $\omega_d$ which has cone singularities with cone angle $2\pi(1-\frac 1d)$ along $B_0$. In other words, $\omega_d$ can be desingularized by taking the ramified cover $(z_1, \underline z)\to (z_1^d, \underline z)$ defined on the weakly pseudoconvex so-called 
\emph{Th\"ullen domain}
$\Omega_d:=\{|z_1|^{2d}+|\underline z|^2<1\}\subset \mathbb C^n$. The metric $\omega_d$ is invariant under the automorphisms of $B$ preserving $B_0$ and hence it descends to $\Gamma_0 \backslash B$ where $\Gamma_0<\Gamma$ is the stabilizer of $B_0$ inside $\Gamma$, and we have $\Gamma_0\backslash B_0=D$.
%of course $\Gamma_0$ has infinite index.  
The desingularization of the metric $\omega_d$ on $\Gamma_0\backslash B$ serves as a model for the 
K\"ahler Einstein metric near the divisor $D\subset M$ along which a branched covering is taken.

\emph{Step 2. Computing the curvature of $\omega_d$.}

\noindent
A large part of the article is devoted to analyzing the model orbifold metric $\omega_d$ on the ball $B$, or rather its desingularization on the Th\"ullen domain $\Omega_d$. Such an investigation was carried out by Bland \cite{Bl86},
but his results are not strong enough for our needs. Our approach is completely different and based on the 
observation  that the behavior of $\omega_d$ is fully determined by 
a well-chosen real valued function solving a second order 
ordinary differential equation, cf Theorem~\ref{thm:ode}. This leads to 
%Although the ODE may not be solved explicitly, one can still derive a rather precise qualitative behavior of the solution. The %outcome is that one can find 
explicit negative bounds for the sectional curvature of $\omega_d$ described in
Theorem~\ref{negativeonreal} and exponential convergence of $\omega_d$ 
to the complex hyperbolic metric $\omega_B$ as the distance to $B_0$ goes to $+\infty$, which is  
formulated in Theorem~\ref{thm:comparison}.

\emph{Step 3. Gluing $\omega_d$ to the hyperbolic metric.}

\noindent
One would like to glue $\omega_d$ on a tubular neighborhood $U$ of $D\subset M$ to the complex hyperbolic metric $\omega_B$ on $M\setminus U$. 
This is of course always possible, but unless the two metrics match very well in the gluing zone, the resulting metric will no longer have good curvature properties there. Controlling the glued metric requires a large collar size 
of the divisor in the arithmetic manifold as this will guarantee that the gluing metric is close to the ball metric 
on the gluing zone. That one can find Stover-Toledo manifolds obtained by a covering branched 
along a {\it single connected} divisor with arbitrarily large collar size is 
shown in Section \ref{sec ST}, see Theorem~\ref{mainapp}. It involves among other things of subgroup separability of stabilizers of 
hyperplanes in arithmetic lattices in ${\rm PU}(n,1)$ of simple type.
%The key input is a result of Stover and Toledo which, after some nontrivial additional work, allows us to get the following.  Given %an arbitrary large number $R$, one can find a large unramified cover $p_R:M_R\to M$  on which the inverse image of $D$ has a collar %of size at least $R$. That is, up to replacing $M$ with $M_R$, one can assume that there exists a neighborhood of $p_R^{-1}%(D)\subset M_R$ of size at least $R$ (measured with respect to the hyperbolic metric) which embeds in $\Gamma_0\backslash B$.
%At this stage, one needs to pick out a connected component $D_R$ of $p_R^{-1}(D)$ and choose $U_R$ a neighborhood of 
%$D_R$ with collar size at least $R$. By  the results from Step 2, one can glue $\omega_d$ near $U_R$ to $\omega_B$ on the complement %of $U_R$ in $M_R$ and obtain a new Kähler metric $\omega_R$ on $M_R$ with negative curvature.  

\emph{Step 4. Deforming to the Kähler--Einstein metric.}

\noindent
As the collar size $R$ of the neighborhood of the divisor $D$ tends to infinity, the glued metric will be 
arbitrarily close to a K\"ahler Einstein metric. All of them have uniformly bounded geometry. Using standard tools
we find that they can be deformed to K\"ahler Einstein orbifold metrics with controlled negative 
curvature provided that 
$R$ is sufficiently large. The desingularization of these K\"ahler Einstein orbifolds in covers 
branched along the singular divisors of the metrics provide the 
examples in the main Theorem.
%
%By construction $\omega_R$ gets closer and closer to being Einstein as $R\to +\infty$. One can prove  that $K_{M_R}+(1-\frac 1d)D_R$ is always ample (see Lemma~\ref{ampleness}) so that the orbifold version of Aubin-Yau's theorem yields the existence of a unique orbifold Kähler--Einstein metric $\omega_{\rm KE, d,  R}$ on the pair $(M_R, (1-\frac 1d)D_R)$. The name of the game is now to prove that $\omega_{\rm KE, d, R}$ is $C^2$ close to $\omega_R$ as $R\to +\infty$. This is achieved by suitably using standard tools adapted to the negative curvature setting (e.g. maximum principle, Yau's generalized Schwarz lemma) and relying on the fact that the metrics at stake have bounded geometry. At this stage, one has only produced an orbifold KE metric with negative curvature. But relying once again on the result of Stover and Toledo, one can find for $R$ large enough a finite cover $M_{d,R}\to M_R$ branched at order $d$ along $D_R$. This cover desingularizes $\omega_{\rm KE,d, R}$ and yields the main result. 

\smallskip
\noindent 
{\bf Acknowledgement:} 
We are grateful to Pierre Py who pointed out an error in the first version of this manuscript.

\smallskip
\noindent 
{\bf Funding:} This material is based upon work supported by the National Science Foundation 
under Grant No. DMS-1928930 while the authors were in residence at the Simons Laufer Mathematical 
Science Institute (former MSRI) in Berkeley, California, during the Fall 2024 semester. H.G. is partially supported by the French Agence Nationale de la Recherche (ANR) under reference ANR-21-CE40-0010 (KARMAPOLIS). U.H. is partially supported by the 
DFG Schwerpunktprogramm SPP 2026 Geometry at infinity and the Hausdorff Center Bonn.

\section{Kähler--Einstein metrics on Th\"ullen domains}\label{bounded}

For $n\geq 2$ consider $\mathbb{C}^n$ with the standard coordinates
$(z_1,\dots,z_n)$ and euclidean norm $\vert \,\vert$. %Put $Z_2=(z_2,\dots,z_n)$. 
The unit ball $B$ in $\mathbb{C}^n$ is be defined by 
\[B=\{(z_1,z_2,\dots,z_n)\in \mathbb{C}^n\mid \vert z_1\vert^2+\sum_{i\geq 2}\vert z_i\vert^2<1\}.\]
The group of biholomorphic automorphisms of $B$ is the group ${\rm PU}(n,1)$. The stabilizer 
of the divisor $B_0=\{z_1=0\}$ equals
\[{\rm Stab}_{{\rm PU}(n,1)}(B_0)=P(S^1\times U(n-1,1))={\rm U}(n-1,1).\] The circle group $S^1$ acts
on $B$ by $(e^{i\theta},(z_1,\dots,z_n))\to (e^{i\theta}z_1,\dots,z_n)$, and it is the subgroup of
${\rm Stab}_{{\rm PU}(n,1)}(B_0)$ which fixes $B_0$ pointwise.
More concretely, ${\rm Stab}_{{\rm PU}(n,1)}(B_0)$ is a central extension of ${\rm PU}(n-1,1)$, the group of 
biholomorphic automorphisms of $B_0$, by the circle group $S^1$. 

For $\alpha \in [1,\infty)$ consider the \emph{Th\"ullen domain}
\[\Omega=\Omega_\alpha=\{(z_1,\dots,z_n)\in \mathbb{C}^n\mid 
\vert z_1\vert^{2\alpha} +\sum_{i\geq 2}\vert z_i\vert^2<1\}.\] 
Clearly we have $\Omega_\alpha =B$ for $\alpha =1$, and 
$\Omega_\infty=D\times B_0$, the product of the unit disk $D$ and the ball of dimension $n-1$.
For $\alpha<\infty$ 
the bounded domain $\Omega_\alpha\subset \mathbb{C}^n$ 
is weakly $C^2$-pseudoconvex. Moreover, for $\alpha=d\in \mathbb{N}$, the domain  $\Omega_\alpha$ maps onto the ball $B\subset \mathbb{C}^n$ by the holomorphic map 
\[\Phi_d:(z_1,z_2,\dots,z_n)\to (z_1^d,z_2, \dots,z_n).\]
The map $\Phi_d$ is a covering of degree $d$, 
branched along $B_0$. For arbitrary $\alpha\geq 1$ we can also formally write a 
map $\Phi_\alpha:\Omega_\alpha\to B$, however it is multi-valued. 

The following is due to Naruki \cite{Na68}. It relies on the fact that the 
coordinate projection 
$(z_1,\dots,z_n)\to (z_2,\dots,z_n)$ is a holomorphic fibration with fiber the disk. 

\begin{lemma}[Naruki]\label{automorphismgroup}
There is a central extension $G$ of ${\rm PU}(n-1,1)$ by $S^1$ which 
acts on 
$\Omega_\alpha$ as a group of biholomorphic automorphisms, and complex conjugation
$z\to \bar z$ acts as an antiholomorphic automorphism.  Moreover, one has an isomorphism $G\simeq \mathrm{Aut}(\Omega_\alpha).$
\end{lemma}

Although the statement of the lemma is well known, we provide a sketch of a proof to illustrate
the nature of the action of $G$ on $\Omega_\alpha$ as this will be important in the 
sequel and is not well documented in the literature. The multi-valued map $\Phi_\alpha$ induces
a homomorphism $G\to {\rm U}(n-1,1)$.  

\begin{proof}[Proof of Lemma \ref{automorphismgroup}]  
By the definition of  $\Omega_\alpha$, the circle group $S^1$ of rotations
in the $z_1$-coordinate, defined by 
\[(\theta,(z_1,z_2,\dots,z_n))\to 
(e^{i\theta}z_1,z_2,\dots,z_n),\] 
acts on $\Omega_\alpha$ as a group of biholomorphic automorphisms. The map $\Phi_\alpha$
maps orbits of $S^1$ to orbits of $S^1$, but it does not commute with the $S^1$-action. More precisely,
we have $\Phi_\alpha \circ \theta= \alpha \theta \circ \Phi_\alpha$.

Consider the ball $B_0=\{z_1=0\}\subset B\cap \Omega_\alpha$. 
If $\psi_0$ is an automorphism of $B_0$, 
then there is a nowhere vanishing holomorphic function $\theta$ on $B_0$ so that 
$\psi_0$ extends to an automorphism $\psi$ of $B$ defined by 
$\psi(z_1,z_2,\dots,z_n)=(\theta(z_2,\dots,z_n)z_1,\psi_0(z_2,\dots,z_n))$.
We know that
$\vert \theta\vert^2(\underline z)=\frac{1-\vert \psi_0 (\underline z)\vert^2}{1-\vert \underline z\vert^2}$.

Any choice $\hat \theta$ of a root of $\theta$, that is, a holomorphic 
function $\hat \theta$ with $\hat \theta^{\alpha}=\theta$, defines an automorphism $\hat \psi$
of $\Omega_\alpha$ by
\[\hat \psi(z_1,\dots,z_n)=(\hat \theta(z_2,\dots,z_n)z_1,\psi_0(z_2,\dots,z_n))\]
 so that 
$\Phi_\alpha \circ \hat \psi=\psi\circ \Phi_\alpha$.
If $\tilde \theta$ is another root, then 
$\tilde \theta =e^{ia} \theta$ where $(e^{ia})^\alpha=1$. In other words, any two such 
choices differ by an element of $S^1$. As a consequence, the stabilizer $G$ of 
$B_0$ in the automorphism group of $\Omega_\alpha$ surjects onto the automorphism
group ${\rm PU}(n-1,1)$ of $B_0$, with kernel $S^1$, and this surjection 
induces a homomorphism $G\to {\rm U}(n-1,1)=\mathrm{Stab}_{\mathrm{PU}(n,1)}(B_0)$. 

At this point we have proved that for any $g\in G$, there exists $h_g\in \mathrm{Stab}_{\mathrm{PU}(n,1)}(B_0)$ such that 
\begin{equation}
\label{equivariance}
\Phi_\alpha \circ g = h_g \circ \Phi_\alpha.
\end{equation}
%
%$\psi$ acts on the ball $B$ 
%Put $z_1=x_1+iy_1$ for $x_1,y_1\in \mathbb{R}$. 
%Any element $z\in B\setminus B_0$ is the image under $\Phi_\alpha$ of 
%a unique point $w\in \Omega_\alpha$ with $\arg(w)\in [0,2\pi/\alpha)$ where the argument is 
%taken of the first coordinate and such that
%$0$ corresponds to $y_1=0$. In other words, the restriction of 
%$\Phi_\alpha$ to $\{\arg(w)\in (0,2\pi/\alpha)\}$ is a biholomorphism onto its image, which 
%is the open dense 
%${\rm PU}(n-1,1)$-invariant subset $\{\arg(u)\in (0,2\pi)\}$ of $B\setminus B_0$.  
%
%Via this identification, the group ${\rm PU}(n-1,1)$ acts on the domain 
%$\{\arg(w)\in (0,2\pi/\alpha)\}$ as a group of biholomorphic automorphisms. 
%
%As this action is compatible with the $S^1$-actions on $\Omega_\alpha$ and 
%$B$, it extends to an action on $\Omega_\alpha-B_0$ by biholomorphic
%transformations. This action then extends to an 
%action on $\Omega_\alpha$ by Hartog's theorem.  

That complex conjugation is an antiholomorphic automorphism of $\Omega_\alpha$ is immediate from the 
definition.

In order to prove that $G$ is isomorphic to the automorphism group of $\Omega_\alpha$ we need to show that any element $g\in \mathrm{Aut}(\Omega_\alpha)$  fixes $B_0$. This can be seen as follows. Up to composing with the lift to $G$ of a suitable element in $\mathrm{PU}(n-1, 1)$, one can assume that $g$ fixes the origin. By Cartan's theorem, $g$ has to act on $\Omega_\alpha$ by a linear transformation of $\mathbb C^n$. Now $g$ fixes $\partial \Omega_\alpha$, and it also fixes the non strictly pseudoconvex locus of the latter, which is exactly $B_0\cap \partial \Omega_\alpha \simeq \{\underline z\in \mathbb C^{n-1}; |\underline z|=1\}$ since $\alpha>1$. If $g$ is represented by the matrix $(a_{ij})$ then we must have $\sum_{j=2}^n a_{1j} z_j=0$ for any $\underline z=(z_2, \ldots, z_n) \in S^{2n-3}$ hence for any $\underline z \in \mathbb C^{n-1}$. That is, we have $g(B_0)\subset B_0$, and equality follows. 
\end{proof}

Since $\Omega_\alpha$ is weakly $C^2$ pseudoconvex, it follows from the work of Cheng and Yau \cite{CY80} that $\Omega_\alpha$ admits a unique complete Kähler--Einstein metric.

\begin{theorem}[Theorem 7.5 of \cite{CY80}]\label{chengyau}
There exists a unique complete Kähler--Einstein metric $g_\alpha$ on $\Omega_\alpha$ with Einstein constant $-(2n+2)$. In particular, $g_\alpha$ is 
invariant under the group $G$ of biholomorphic transformations and under complex conjugation. 
\end{theorem}

\begin{proof}
Since $\Omega_\alpha$ is weakly $C^2$ pseudoconvex, the existence of \emph{some} complete invariant Kähler--Einstein metric $\omega_\alpha$ on $\Omega_\alpha$ is Theorem 7.5 of \cite{CY80}, which however does not state uniqueness explicitly. Uniqueness is a classic consequence of Yau's Schwarz lemma and his generalized maximum principle. Indeed, Theorem~3 in \cite{Y78} shows that if $\omega$ and $\omega'$ are two complete Kähler--Einstein metrics with the 
same Einstein constant $c<0$, then the ratio $F:=\log\Big(\frac{\omega'^n}{\omega^n}\Big)$ is globally bounded. Finally, since $dd^c F=-c(\omega'-\omega)$, applying the maximum principle \cite{Y75} to $\pm F$ yields $F\equiv 0$, hence $\omega'=\omega$.  

The invariance of the associated Riemannian metric $g_\alpha$ under the group of holomorphic automorphisms $G$ is a direct consequence of the 
invariance of $\omega_\alpha$ and the fact that the Riemannian metric can be recovered from the Kähler form. 
Now, if $\varphi$ is the diffeomorphism of $\Omega_\alpha$ induced by complex conjugation and $J$ is the complex structure, we have $\varphi J=-J\varphi$. This implies that $J$ preserves $\varphi^*g_\alpha$ and that $\varphi^*J=-J$. In particular, we have $\nabla^{\varphi^*g_\alpha}J=0$ so that the positive real  $(1,1)$-form associated to $(\varphi^*g_\alpha,J)$ (which is nothing but $-\varphi^*\omega_\alpha$) is closed; thus it is Kähler--Einstein. By uniqueness, it must coincide with $\omega_\alpha$. This implies that $\varphi^*g_\alpha=g_\alpha$. 
\end{proof}

\begin{remark}[Comparison with the Bergman metric]
    \label{Bergman} 
    The bounded domain $\Omega_\alpha$ can be equipped with the \emph{Bergman metric} $h_\alpha$. It was proved in Theorem~3 of \cite{AS83}
that the holomorphic sectional curvature of the Bergman metric $h_\alpha$ is contained 
in an interval $[-b^2,-a^2]$ for some $0<a<b<\infty$ not depending on $\alpha$. In particular, it follows from Theorem 4.4 of \cite{CY80}
that $g_\alpha$ is bi-Lipschitz equivalent to $h_\alpha$. 
\end{remark}

The invariant Kähler--Einstein metric $g_\alpha$ on 
$\Omega_\alpha$ with Einstein constant $-(2n+2)$ whose existence
was pointed out in Theorem \ref{chengyau}
was studied by Bland \cite{Bl86} who proved that its 
sectional curvature is negative. 
The goal of this section is to improve Bland's result and establish the following 
explicit description of $g_\alpha$. 

\begin{theorem}\label{negke}
The complete $G$-invariant Kähler--Einstein metric $g_\alpha$
on $\Omega_\alpha$ has the
following properties.
\begin{enumerate}
\item The divisor $B_0$ is totally geodesic.
\item The sectional curvature of $g_\alpha$ is contained in an interval of the form
$[-2n-2,-a_\alpha^2]$ for some $0<a_\alpha \leq 1$. 
\item The holomorphic sectional curvature ranges in $[-2n-2,-4]$.
\item For $d\in \mathbb{N}$, it holds $(\Phi_d^*g_1- g_\alpha)(z)\to 0$, 
exponentially with the distance 
of $\Phi_d(z)$ from $B_0$.
\end{enumerate}
\end{theorem}

The last property of the theorem will be made more precise during the course of the proof. 

Bland does not
establish the asymptotic behavior of the metric transverse to the divisor (part (4) of
the above theorem), which is a crucial ingredient in the proof of our main result. 
This property as well as the explicit description of the curvature does not seem
obvious from his formulas. 

The remainder of this section is devoted to the proof of Theorem \ref{negke}. Our argument
is different from Bland's approach. Its main idea is to reduce the study of the 
metric to an ordinary differential equation which can be solved fairly explicitly. 
The proof is spread over four  subsections. In the first subsection we collect 
some properties of arbitrary invariant K\"ahler metrics on $\Omega_\alpha$, and we use this
in the second subsection to obtain some first information on the curvature tensor of 
such metrics. These results equally hold true for the Bergman metric. 
In the third subsection we turn to the Kähler--Einstein metric and set up an ordinary differential equation 
whose solutions describe the metric fairly explicitly as described in the theorem.
The curvature computation is contained in the forth subsection.

 \subsection{Geometric properties of $G$-invariant K\"ahler metrics on $\Omega_\alpha$}
 
In this subsection we consider an arbitrary complete K\"ahler metric $g$ on $\Omega_\alpha$ which 
is invariant under the group $G$ and under complex conjugation. Examples
we have in mind are the Bergman metric of $\Omega_\alpha$ and the invariant Kähler--Einstein metric $g_\alpha$
whose existence was shown in Theorem \ref{chengyau}.
We establish some general geometric properties with the goal to reduce curvature computations 
to the computation of the curvature of some specific planes in the tangent bundle of $\Omega_\alpha$.

A \emph{standard totally real} plane in $\Omega_\alpha$ is 
the intersection of $\Omega_\alpha$ with $\{z\in \Omega_\alpha\mid 
z_i=0 \text{ for }i\geq 3 \text{ and } z-\bar z=0\}$. A \emph{totally real plane} in $\Omega_\alpha$ is 
the image of the standard totally real plane under an element of the group $G$.
We have

\begin{lemma}\label{total}
\begin{enumerate}
\item
The isometry group of $g$ is of 
cohomogeneity one. 
\item 
The disk $D=\{z_i=0 \text{ for }i\geq 2\}$ and the standard 
totally real plane are totally geodesic.
\item  
The ball $B_0=\{z_1=0\}$ is
totally geodesic, and the restriction of $g$ to $B_0$ is up to a constant factor the complex hyperbolic 
metric.
\end{enumerate}
\end{lemma}
\begin{proof}
As the metric $g$ is invariant under the group $G$ and the generic orbit of this
group on the ball $B$ and hence on $\Omega_\alpha$ by equivariance is of real codimension one, 
the action of the isometry group of $g$ is of cohomogeneity one showing (1) of the lemma.

Since the disk $D$ is the fixed point set of the holomorphic involution 
\[(z_1,z_2,\dots,z_n)
\to (z_1,-z_2,\dots,-z_n)\] which is an element of the group $U(n-1)\subset G$ 
(the symmetric involution at the point $0\in B_0$)  
and hence
an isometry for $g$, the disk $D$ is totally geodesic.

Similarly, the ball $B_0$ is the fixed point set of the holomorphic reflection
\[(z_1,z_2,\dots,z_n)\to (-z_1,z_2,\dots, z_n)\in S^1\] and hence it is totally geodesic. 
 Since the restriction of $g$ to $B_0$ is invariant under $G$ and since 
 $G$ projects to ${\rm PU}(n-1,1)$ and hence acts transitively on the unit tangent bundle of $B_0$ for the complex hyperbolic metric,
 the restriction of $g$ to $B_0$  is a multiple of the complex hyperbolic metric which 
 establishes part (3) of the lemma.

Now the subspace $V=\{z_i=0 \text{ for all }i\geq 3\}$ also is the fixed point set of 
a holomorphic isometry $(z_1,z_2,z_3,\dots,z_n)\to (z_1,z_2,-z_3,\dots, -z_n)$
of $g$ contained in the group $G$ 
and hence it is totally geodesic. Furthermore,
the set $\{\Im z_i =0, i\geq 1\}$ is the fixed point set of complex conjugation and hence it is 
totally geodesic. As the intersection of two totally geodesic subspaces is totally geodesic, 
the standard real plane is totally geodesic. By invariance, the same then holds true for 
any of its images under the isometry group of $g$.
This completes the proof.
 \end{proof}
 
Consider a point $z\in D$. The real tangent 
space of $\Omega_\alpha$ at $z$ decomposes as 
\[T_z\Omega_\alpha=T_zD\oplus T_zD^\perp\] 
where $T_zD^\perp$ is the orthogonal complement of $T_zD$. Since 
$g$ is K\"ahler and $T_zD$ is 
invariant under the complex structure $J$, viewed as a tensor field on $\Omega_\alpha$, 
the same holds true for $T_zD^\perp$. 
 
The group $G$ of biholomorphic transformations 
of $\Omega_\alpha$ 
preserves the totally geodesic
submanifold $B_0$. Then it also preserves the level sets of the distance function to $B_0$ for
the $G$-invariant K\"ahler metric $g$.
%and the same holds true for the action of the circle group $S^1$. 

\begin{lemma}\label{orthogonal}
A level set of the distance function from $B_0$ is the preimage under $\Phi_\alpha$ of a level
set of the distance function from $B_0$ in $B$ equipped with the complex hyperbolic metric $g_1$. 
The group $G$ of automorphisms of $\Omega_\alpha$
acts transitively on any such level set. 
\end{lemma}
\begin{proof} 
The action of  
$G$ on the preimage under $\Phi_\alpha$ of the boundary of a tubular
neighborhood of the divisor $\{z_1=0\}$ in the ball $B$ is transitive, 
and an orbit is connected and separates $\Omega_\alpha$ into two components, one of which 
contains $B_0$. As $B_0$ can be connected to any point in $\Omega_\alpha$ by a minimal geodesic, 
we conclude that such an orbit equals the boundary 
$N(r)$ of the tubular neighborhood of radius $r\geq 0$ about $B_0$. As a consequence, the action of 
$G$ on $N(r)$ 
is transitive.  
\end{proof}

\subsection{The curvature operator of an invariant K\"ahler metric}

In this subsection we investigate the curvature tensor
$R$ of an arbitrary 
$G$-invariant K\"ahler metric $g=\langle , \rangle$ on $\Omega_\alpha$. 
It can be viewed as a section of the tensor bundle ${\rm Sym}(\Lambda^2T\Omega_\alpha)$
of symmetric linear maps 
$\wedge^2 T\Omega_\alpha\to \wedge^2 T\Omega_\alpha$ (all the vector spaces here are viewed as 
real vector spaces). 
For $z\in D$ the stabilizer ${\rm U}(n-1)\subset G$ 
 of $z$ in the isometry group of $g$
acts on $T_z\Omega_\alpha$ as a group of isometries commuting with the complex
structure. This action induces a representation of ${\rm U}(n-1)$ on $\wedge^2T_z\Omega_\alpha$
by linear isometries for the induced metric. 
The representation decomposes into irreducible components. The curvature tensor
$R$ is equivariant under the action of ${\rm U}(n-1)$ and hence it preserves the union of all 
linear subspaces of $\wedge^2T_z\Omega_\alpha$ belonging to isomorphic irreducible components. 
This leads to the following statement. 

\begin{lemma}\label{representation}
%Suppose that $n\geq 3$. 
\begin{enumerate}
\item 
Let 
$v_1,v_2=Jv_1$ be an orthonormal basis of $T_zD$; then $v_1\wedge v_2$ is an eigenvector 
for $R$.
\item Let $\{v\wedge w\mid v\in T_zD,$ and $w\in T_zD^\perp\}$; then
$v\wedge w$ is an eigenvector for $R$. The eigenvalue does not depend on $v,w$. 
\item The subspace $\wedge^2 T_zD^\perp$ is invariant under $R$.
%and restricts to a multiple of the curvature tensor 
%of the ball of dimension $n-1$ as an ${\rm U}(n-1)$-space. 
\end{enumerate}
\end{lemma}
\begin{proof} The representation of ${\rm U}(n-1)$ on $T_z\Omega_\alpha$ decomposes into irreducible
components as follows. 
The restriction of ${\rm U}(n-1)$ to the tangent space $T_zD$ of $D$ is the trivial 
representation, 
while the restriction of ${\rm U}(n-1)$ to $T_zD^\perp$ is the standard 
representation of ${\rm U}(n-1)$ on a complex vector space of dimension $n-1$. This representation is 
well known to be irreducible (for example via transitivity of the action of 
${\rm U}(n-1)$ on the unit sphere in $\mathbb{C}^{n-1}$).

From this information, we can compute the irreducible components of the action of 
${\rm U}(n-1)$ on $\wedge^2T_z\Omega_\alpha$.
Observe that $\wedge^2T_z\Omega$ is a direct sum of subspaces 
\[\wedge^2T_z\Omega_\alpha=A_1\oplus A_2\oplus A_3\]
where $A_1=\wedge^2 T_zD$, 
$A_2=T_zD\wedge T_zD^\perp$ and $A_3=\wedge^2 T_zD^\perp$.
This decomposition is invariant under the action of ${\rm U}(n-1)$ and orthogonal with respect 
to the inner product induced by $g$.
The real dimension of $A_2$ equals $2(2n-2)$.

The line $A_1$ is contained in the fixed point set for the action of ${\rm U}(n-1)$, that is,
it is contained in a copy of the trivial representation.  
%and its is characterized by this 
%property. Thus this line is invariant under the curvature tensor $R$. This shows the first
%part of the lemma.

For a unit vector $v\in T_zD$, the 
action of ${\rm U}(n-1)$ on the real $2n-2$-dimensional subspace 
$A_2(v)={\rm span}\{v\wedge w\mid w\in T_zD^\perp\}\subset A_2$ of $A_2$
can be identified with the standard action of ${\rm U}(n-1)$ on 
$\mathbb{C}^{n-1}$, viewed as a real vector space. 
Thus $A_2(v)$  is invariant under ${\rm U}(n-1)$, and the
restriction of the representation to this subspace is irreducible.  
Now the image of $A_2(v)$ under the complex structure $J$ is the 
subspace $A_2(Jv)$, and we have 
$A_2=A_2(v)\oplus A_2(Jv)$ as ${\rm U}(n-1)$-spaces. 
Thus as an $U(n-1)$-representation, $A_2$ is a direct sum of two standard
representations of ${\rm U}(n-1)$ on $\mathbb{C}^{n-1}$. 

On the other hand, the representation of ${\rm U}(n-1)$ on $\wedge^2T_zD^\perp$
is the standard representation of ${\rm U}(n-1)$ on the exterior product 
$\wedge^2\mathbb{C}^{n-1}$, where we view $\mathbb{C}^{n-1}$ as a real vector space. 
The complex structure $J$ acts on $\wedge^2\mathbb{C}^{n-1}$ as an involution. 
Since ${\rm U}(n-1)$ commutes with $J$, it preserves the eigenspaces $V_{\pm}$ for $J$ with respect
to the eigenvalues $\pm 1$. 

The eigenspace $V_+$ for the eigenvalue one is the kernel of the 
$\mathbb{R}$-linear map 
$\Lambda:\wedge^2 \mathbb{C}^{n-1}\to \wedge^2_{\mathbb{C}}\mathbb{C}^{n-1}$ 
obtained by extension of scalars. Here the vector space on the right hand side
is the second exterior power of the complex vector space $\mathbb{C}^{n-1}$.
The vector space $V_+$ 
is spanned by elements of the form $v\wedge Jv=-Jv\wedge v$ for 
$v\in T_zD^\perp$. Since the center $S^1$ of $U(n-1)$ which contains the complex structure 
acts trivially on $V_+$ but it does not act trivially on the standard representation 
space $A_2(v)$, there can not be a copy of the standard representation in $V_+$. 

The representation of ${\rm U}(n-1)$ on $V_-$ is the  
representation of ${\rm U}(n-1)$ on the complex vector space $\wedge^2_{\mathbb{C}}\mathbb{C}^{n-1}$, 
viewed as a vector space over $\mathbb{R}$,  
and hence it 
irreducible, with highest weight different from the weight of the standard 
representation. As a consequence, 
$A_2$ equals the union of those irreducible components for the ${\rm U}(n-1)$-representation
on $\wedge^2T_z\Omega_\alpha$
which are isomorphic to the standard representation of ${\rm U}(n-1)$ on $\mathbb{C}^{n-1}$.

Since the curvature tensor $R$ commutes with the action of ${\rm U}(n-1)$
on $\wedge^2T_z\Omega_\alpha$ , the vector space 
$A_2$ is invariant. But $R$ also commutes with the complex structure $J$ which maps 
$A_2(v)$ to $A_2(Jv)$ and therefore
$A_2$ is an eigenspace for $R$. Moreover, $R$ preserves $A_1\oplus A_3$ since $R$ is symmetric
and the decomposition $\wedge^2T_z\Omega_\alpha=A_2 \oplus (A_1\oplus A_3)$ is orthogonal.

We use this to establish that 
for $v\in T_zD$, the vector $v\wedge Jv$ is an eigenvector for $R$.
Namely,  as $R$ is a symmetric operator and the decomposition 
$A=A_2\oplus (A_1\oplus A_3)$ is orthogonal, with $A_2$ invariant under $R$, 
if $v\wedge Jv$ is not an eigenvector for $R$ then there are 
$w_1\not=w_2\in T_zD^\perp$ orthogonal so that 
\[\langle R(v,Jv)w_1,w_2\rangle \not=0.\]
However, by the Bianchi identity, we have 
\[R(v,Jv)w_1+R(Jv,w_1)v+R(w_1,v)Jv=0.\]
But $Jv\wedge w_1\in A_2, w_1\wedge v\in A_2$ and 
$v\wedge w_2\in A_2$ is orthogonal to $Jv\wedge w_1$ and 
$w_1\wedge v$ is orthogonal to $Jv\wedge w_2$. Since $A_2$ is an eigenspace for $R$
for a fixed real eigenvalue, this implies that $\langle R(Jv,w_1)v,w_2 \rangle =0=
\langle R(w_1,v)Jv,w_2\rangle =0$ and hence 
$\langle R(v,Jv)w_1,w_2\rangle =0$, a contradiction to the assumption that
$v\wedge Jv$ is not an eigenvector for $R$.

As a consequence, 
the decomposition $A=A_1\oplus A_2\oplus A_3$ is invariant under $R$.
Furthermore, $A_1$ and $A_2$ are eigenspaces for $R$.
This completes the proof of the lemma.
%and it satisfies the Bianchi identities. 
%
%Consider the orbit $Q$ for the action of the group $U(n,1,1)$ through $z$. The stabilizer of 
%$z$ equals the subgroup $U(n-1)$. 
%equals a multiple of the curvature tensor of the ball of dimension $n-1$, which is
%again a consequence of invariance under $J$ and the Bianchi identities. \marginpar{H. A few more details here maybe, or a reference?}
\end{proof}

%Recall that $J$ denotes the complex structure in $T\Omega_\alpha$. 

\begin{corollary}\label{negative}
The curvature of $g$ is negative if and only if 
the following three conditions are satisfied.
\begin{enumerate}
\item The Gauss curvature of the disk $D$ is negative.
\item The curvature of the standard totally real plane is 
negative.
\item For every $z\in D$ there exists a $J$-invariant 
plane in $T_z^\perp D$ whose curvature is negative.
\end{enumerate}
%If $g$ is moreover Einstein, then the third condition follows from the first two.
\end{corollary}
\begin{proof}
Clearly the conditions in the corollary are necessary. 
We only show that they are sufficient if we replace assumption (3) by the 
following stronger assumption. 

\noindent
$(3^\prime)$
For every $z\in D$ the curvature of every plane in $T_z^\perp D$ is negative.

In the proof of Theorem \ref{thm:ode} below we shall establish that (3) implies $(3^\prime)$, cf Remark~\ref{3 and 3'}. 

To show that the assumptions (1), (2), ($3^\prime$) imply negative curvature  of $g$ 
note first that by invariance under the isometry group of $\Omega_\alpha$, 
it suffices to verify that the curvature is negative at every point $z\in D$.
Using the assumptions in the lemma, 
it suffices to compute the curvature of a plane 
spanned by 
$u_1=v_1+w_1,u_2=v_2+w_2$ with $v_i\in T_zD$ and $w_j\in T_zD^\perp$ and such that
$v_1\not=0, w_2\not=0$. We allow that 
either $v_2=0$ or $w_1=0$. We may also assume that $\langle v_1,v_2\rangle =0$. 

Now $u_1\wedge u_2=v_1\wedge v_2+v_1\wedge w_2+w_1\wedge v_2+w_1\wedge w_2$. By 
Lemma \ref{representation} and orthogonality of the 
decomposition of $\wedge^2T_z\Omega_\alpha$ into eigenspaces 
for $R$, we compute 
\begin{align*}
\langle R(u_1,u_2)u_2,u_1\rangle &=\langle R(u_1\wedge u_2),u_2\wedge u_1\rangle \\
&=
\langle R(v_1,v_2)v_2,v_1\rangle +
\langle R(w_1,w_2)w_2,w_1\rangle \\
&+\langle R(v_1,w_2)w_2,v_1\rangle 
+\langle R(w_1,v_2)v_2,w_1\rangle .\end{align*}
By the assumption in the corollary, this  
is a sum of non-positive terms, with at least one term negative. 
This completes the proof of the lemma.
\end{proof}

\subsection{An ordinary differential equation for the Kähler--Einstein metric}
\label{sec model KE}

From now on we consider the $G$-invariant Kähler--Einstein metric 
$g=g_\alpha$ on $\Omega_\alpha$ 
whose existence was shown in Theorem \ref{chengyau}.

By Lemma \ref{orthogonal}, for $r>0$ 
the level surface $N(r)$ of level $r$ 
for the distance to the hyperplane $B_0$ 
is a real hypersurface in the complex manifold $\Omega_\alpha$ which 
is invariant under the action of the group $G$, and this action is transitive
on $N(r)$. 

The
maximal $J$-invariant subbundle ${\cal D}$ of $TN(r)$ is a smooth subbundle of $N(r)$ of codimension 
one. The fiber ${\cal D}_z$ 
of ${\cal D}$ at a point $z\in D\cap N(r)$ is invariant under the action of the group 
${\rm U}(n-1)$, and ${\rm U}(n-1)$ acts transitively on the sphere of unit tangent vectors in ${\cal D}_z$. 
Since the action of $G$ preserves 
$N(r)$ and is transitive on $N(r)$, 
it follows that $G$ acts 
transitively on the sphere bundle of unit tangent vectors in ${\cal D}$.

Let $\Pi_r:N(r)\to B_0$ be the shortest distance projection. 
Since by Lemma \ref{total}, the disk $D=\{z_i=0 \text{ for }i\geq 2\}$ is 
totally geodesic and its tangent space at $0$ is the orthogonal complement of 
$TB_0$, the fiber of $\Pi_r$ over $0$ is an $S^1$-orbit in $D$. 
As the distance to $B_0$ is $G$-invariant, the projection 
$\Pi_r$ is equivariant with respect to the $G$-action. Thus 
the fiber of $\Pi_r$ over every point $p\in B_0$ 
is an orbit of the $S^1\subset G$-action, and 
the differential of $\Pi_r$ maps the bundle ${\cal D}$ equivariantly onto the tangent bundle of $B_0$. 

Since the action of $G$ on the unit sphere bundle in ${\cal D}$ is transitive and the
action of $G$ on the unit sphere bundle of $TB_0$ is transitive as well, 
there exists a constant 
$f_\alpha(r)>0$ so that the restriction of $d\Pi_r$ to any fiber of ${\cal D}$ is a homothety of the metric
tensors with factor $f_\alpha(r)^{-2}$.  
Here we equip $B_0$ with the metric $g_0$ of constant holomorphic sectional curvature $-4$ and hence
$f_\alpha(0)^{-2}g_0$  is the restriction of the metric $g$ to $B_0$, where 
$f_\alpha(0)$ may be different from $1$. 

This discussion is valid for any $\alpha\geq 1$, and the function $f_\alpha$ depends on 
$\alpha$. 
The following is the main result of this section and our main technical tool.

\begin{theorem}\label{thm:ode}
For $\alpha\in [1,\infty)$
the function $f_\alpha$ is a solution of the differential equation
\begin{equation}\label{ode} \frac{f^{\prime\prime}}{f} + n \frac{(f^\prime)^2}{f^2}+ n \frac{1}{f^2}=n+1
\end{equation}
with initial condition $f_\alpha^\prime(0)=0$ and $f_\alpha(0)\in (\sqrt{\frac{n}{n+1}},1]$.
The map $\alpha\to f_\alpha(0)$ is a decreasing homeomorphism 
$[1,\infty)\to (\sqrt{\frac{n}{n+1}},1]$. 
\end{theorem}

Note that the solution of (\ref{ode}) for the initial condition $f(0)=1,f^\prime(0)=0$ 
is the function $f(t)=\cosh(t)$ which describes the metric of constant holomorphic
sectional curvature $-4$ on the ball, and the solution with initial condition
$f(0)=\sqrt{\frac{n}{n+1}}$, $f^\prime(0)=0$ is the constant function 
which can be thought of as belonging to a product metric, corresponding to the 
case $\alpha=\infty$. 

%The proof of Theorem \ref{thm:ode} is done in three steps. In a first step, we show that
%the growth function $f=f_\alpha$ for the 
%invariant Kähler--Einstein metric $g$ on $\Omega_\alpha$ is a solution of 
%the differential equation (\ref{ode}) on the maximal interval $[0,T)$ for which 
%the derivative $f^\prime$ is positive, with $f^\prime(0)=0$. 
%In a second step, we analyze solutions of the equation with initial conditions
%$f(0)\in (\sqrt{\frac{n}{n+1}},1]$, $f^\prime(0)=0$ and show that they give rise
%to Kähler--Einstein metrics with controlled geometry. In a third step we then 
%verify that indeed, the invariant Kähler--Einstein metric on $\Omega_\alpha$ 
%corresponds to such a solution. 

\begin{proof}[Proof of Theorem \ref{thm:ode}]
Let for the moment $g$ be any K\"ahler metric on $\Omega_\alpha$ which is 
invariant under the group $G$ of biholomorphic automorphisms of 
$\Omega_\alpha$ and under complex conjugation. 
 
The holomorphic disk $D=\{z_i=0\text{ for }i\geq 2\}$ is totally geodesic  for $g$,
and the same holds true for any of its images under the group
$G$. Thus if we denote by $\xi$ the outer normal field of 
the distance hypersurface $N(t)$, then 
as for a point $z\in D$ the vector 
$\xi_z$ is tangent to $D$, we have 
$J\xi_z\in TD\cap TN(t)$. As $D$ is totally geodesic, this implies that  
$J\xi$ is a principal vector field for the hypersurface $N(t)$.
%Let $\lambda$ be the corresponding principal curvature. 
Similarly, since the group $G$ acts transitively on the sphere subbundle of 
the complex subbundle
${\cal D}=(J\xi)^\perp \subset TN(t)$, 
the bundle ${\cal D}$ is a principal bundle for $N(t)$ by equivariance. 
Put $f=f_\alpha$
for simplicity of notation.

\noindent
{\bf Claim 1:} The principal curvature $\lambda$ of ${\cal D}$ equals
\begin{equation}\label{principal}\lambda=-\frac{d}{dt}f(t)/f(t).\end{equation}

\smallskip \noindent
\emph{Proof of Claim 1:}
The proof of the claim is standard. 
Let $\gamma:(-\infty,\infty)\to D$ be a geodesic through $\gamma(0)=0$ and parameterized by arc length.
Choose a one-parameter group $\phi_s$ of transvections in ${\rm PU}(n-1,1)$ so that $s\to \phi_s(0)$ is a geodesic in 
$(B_0,g)$ parameterized by arc length. Note that this makes sense as 
$B_0\subset \Omega_\alpha$ is a totally geodesic hypersurface 
by Lemma \ref{total} and by invariance, the restriction of 
$g$ to $B_0$ is a multiple of the standard metric on the ball. 

This one-parameter group of transvections
defines a one-parameter group in $G$, denoted by the same symbol $\phi_s$, 
so that 
the image of the map $(s,t)\in \mathbb{R}^2\to \alpha(s,t)=\phi_s(\gamma(t))\subset \Omega_\alpha$ 
is a totally real plane $H$ containing $\gamma$. 
Lemma \ref{total} shows that $H$ is totally geodesic, and it is foliated by the geodesics $\phi_s(\gamma)$. 
The vector field 
$Y(t)=\frac{\partial}{\partial s} \phi_s(\gamma(t))\vert_{s=0}$ is a normal Jacobi field along $\gamma$, and 
as $Y$ is orthogonal to $\gamma$ and tangent to $H$, it is a section of ${\cal D}\vert \gamma$.  
Thus we have
\[\vert Y(t)\vert =f(t)/f(0).\]

Let $h$ be the second fundamental form of the hypersurface $N(t)$ with respect to the 
outer normal field $\xi$ of $N(t)$. 
We have to show that 
\[h(Y(t),Y(t))/\vert Y(t)\vert^2= -\frac{d}{dt}f(t)/f(t).\]
Namely, 
we know that 
\[h(Y(t),Y(t))=\langle \nabla_{Y(t)}Y(t), \xi\rangle=
\langle \frac{\nabla}{ds}\frac{\partial}{\partial s}\alpha(s,t),\xi\rangle\]
where $\nabla$ denotes the Levi Civita connection of $g$. 
Using the fact that $Y(t)\perp \gamma^\prime(t)$ and that $\xi=\frac{\partial}{\partial t}\alpha(s,t)$,
we compute
\begin{align*}
 \langle \frac{\nabla}{ds}\frac{\partial}{\partial s}\alpha(s,t),\xi\rangle 
 &=-\langle \frac{\partial}{\partial s}\alpha(s,t),\frac{\nabla}{d s} \frac{\partial}{\partial t}\alpha(s,t)\rangle\\
 & =-\langle \frac{\partial}{\partial s}\alpha(s,t), \frac{\nabla}{d t} \frac{\partial}{\partial s}\alpha(s,t) \rangle 
 =- \frac{1}{2} \frac{d}{dt} \vert Y(t)\vert^2
  \end{align*} 
 from which the claim follows. 
\hfill $\blacksquare$

\smallskip

Note that the Gauss curvature $K_{\rm tr}(t)$ of the totally geodesic real plane $H$ at $\gamma(t)$ equals 
\begin{equation}\label{cur1}K_{\rm tr}(t)=-\frac{d^2}{dt^2}f(t)/f(t).\end{equation}
Namely, by the Jacobi equation, this curvature equals the quantity
\[- \langle Y^{\prime\prime}(t), Y(t)\rangle/\vert Y(t)\vert^2.\] 

Following p.166 of \cite{KN69}, the curvature tensor $R_0$ of a K\"ahler manifold of 
constant holomorphic sectional curvature $-4$ can pointwise explicitly written 
only in terms of the metric
%\marginpar{H. "depends on the metric and the complex structure \emph{pointwise}" maybe?} 
and the complex structure. Thus it is (formally) defined on any complex vector space with 
a $J$-invariant inner product. In particular, it is defined on a fiber of the bundle 
${\cal D}$. We next compare the restriction of $R$ to $\wedge^2 {\cal D}$ with $R_0$.

\smallskip
\noindent
{\bf Claim 2:} $R\vert_{\wedge^2{\cal D}}=\frac{1}{f(t)^2}(f^\prime(t)^2+1)R_0$.

\smallskip\noindent \emph{Proof of Claim 2.}
Let $A:TN(t)\to TN(t)$ be the shape
operator (or fundamental tensor) of $N(t)$ with respect to $\xi$, defined by 
\[h(X,Y)=\langle AX,Y\rangle= \langle \nabla_{X}\xi,Y\rangle\]
where as before $h$ denotes the second
fundamental form of the hypersurface $N(t)$. 
Claim 1 yields that $A\vert {\cal D}=\lambda {\rm Id}= -(\frac{d}{dt} f(t)/f(t)){\rm Id}$
and hence $h\vert {\cal D}=\lambda \langle , \rangle \vert {\cal D}$. 
Denote by $R^t$ the  curvature tensor of $N(t)$ with respect to the restriction of the metric $g$. 
If $\KN$ denotes the Kulkarni Nomizu product, then it follows from the Gauss Codazzi equations that 
we have 
\[R=R^t- \frac{1}{2} h\KN h.\]
Thus to compute the restriction of the curvature operator 
$R$ to the invariant subbundle $\wedge^2{\cal D}$ it suffices
to compute the curvature operator $R^t$ of $N(t)$.

%By definition, we have 
%\[R^0 \circ d\Pi_r= \frac{1}{f(r)^2} R^r.\]
 Put $U=J\xi$; then $U$ is the normal field to ${\cal D}$ in $TN(t)$. 
Since $g$ is K\"ahler we have 
 \[ \nabla_XU=\nabla_X(J\xi)=J(\nabla_{X}\xi)=JA(X)=PAX\]
 where $P$ is the skew-symmetric $(1,1)$-tensor field on $M$ characterized by
 \[JX=PX+\langle X,U\rangle \xi.\]
This shows that $PA\vert {\cal D}$ is the fundamental tensor of 
 the bundle ${\cal D}$ with respect to the normal field $-U$.
Note that $PU=0, P\vert {\cal D}=J\vert {\cal D}$ and $P{\cal D}={\cal D}$. 
Furthermore, we have 
\begin{equation*}
\nabla_XY-\lambda \langle PX,Y\rangle U \in {\cal D}\end{equation*}
for any sections $X,Y$ of ${\cal D}$ where as before, $\lambda$ is the principal curvature of 
${\cal D}$.
In particular, if $X,Y$ are sections of ${\cal D}$ then as $\nabla$ is torsion free, we have
\begin{equation}\label{fundtensorD}
[X,Y] =Z+2\lambda \langle JX,Y\rangle U\end{equation}
for a section $Z$ of ${\cal D}$. 
%As a consequence, the CR-structure of $N(r)$ is determined by $\lambda$.

As the map $\Pi_t=\Pi\vert N(t)$ restricts to a homothety on ${\cal D}$, with scaling factor 
$f^2(t)$ with respect to the metric $g_0$ on $B_0$,  
the map $\Pi_t:N(t)\to (B_0, f(t)^2 g_0)$ is a Riemannian submersion. 
Thus the formula (\ref{fundtensorD}) 
together with O'Neill's curvature formula for
Riemannian submersions shows that we have
\begin{align*}
\langle R^t(X,Y)Z,W\rangle  &=\frac{1}{f(t)^2}\langle 
R_0(X,Y)Z,W\rangle +\lambda^2 \langle JX,Z\rangle \langle JY,W\rangle  \\ 
&-\lambda^2\langle JY,Z\rangle \langle JX,W\rangle 
+2\lambda^2
\langle JZ,W\rangle \langle JX,Y\rangle.
\end{align*}
On the other hand, we have 
\[-\frac{1}{2}h\KN h(X,Y,Z,W)=\lambda^2 \langle X,Z \rangle 
\langle X,W \rangle -\lambda^2 \langle X,W \rangle \langle Y,Z\rangle.\]
and consequently
\begin{align*}
\langle R(X,Y)Z, W\rangle = &\frac{1}{f(t)^2} \langle R_0(X,Y)Z,W\rangle + \lambda^2 \langle X,Z\rangle \langle Y,W\rangle \\
&-\lambda^2\langle X,W\rangle \langle Y,Z\rangle +\lambda^2 \langle JX,Z\rangle \langle JY,W\rangle \\
&-\lambda^2\langle JY,Z\rangle \langle JX,W\rangle +2\lambda^2 \langle JZ,W\rangle \langle JX,Y\rangle.
\end{align*}

Following p.166f of \cite{KN69}, the above equality shows that
\[\langle R(X,Y)Z,W\rangle =\frac{1}{f(t)^2}\langle R_0(Y,Y)Z,W\rangle +\lambda^2 q\]
where $q$ is the curvature tensor of the complex hyperbolic space with holomorphic 
sectional curvature $-4$. 
As a consequence, 
the restriction of $R$ to $\wedge^2 {\cal D}$ equals 
\begin{equation}\label{curvaturemetric}
R\,\vert_{\wedge^2{\cal D}}=(\lambda^2 +\frac{1}{f^2(t)}) R_0\,\vert_{\wedge^2{\cal D}} .\end{equation}
As $\lambda=f^\prime(t)/f(t)$, we obtain that the multiplicity is 
given by 
\begin{equation*}
\frac{1}{f(t)^2}(f^\prime(t)^2+1)\end{equation*} 
which completes the proof of the claim.
%In particular, the curvature of ${\cal D}$ is negative since $f^\prime(r)>0$.
\hfill$\blacksquare$

By the above computation, 
the value of the Ricci tensor ${\rm Ric}$ on a unit tangent vector $X\in {\cal D}$
equals 
\begin{equation}\label{odegeneral}
-2f^{\prime\prime}(t)/f(t)- 2n\frac{1}{f^2(t)}(f^\prime(t)^2 +1)\end{equation}
since $-2n$ is the Ricci curvature of $B_0$ and $-f^{\prime\prime}/f$ is the Gauss curvature 
of a totally real plane in $\Omega_\alpha$. Here we use that the metric 
$g=\langle ,\rangle$ is K\"ahler
and hence if we denote again by $\xi$ the outer normal field of $N(t)$, then we have 
\[\langle R(X,\xi)\xi,X\rangle =\langle R(JX,J\xi)J\xi,JX\rangle=
\langle R(X,J\xi)J\xi,X\rangle\]
where the last equality follows from ${\rm U}(n-1)$-equivariance of $R$ and the invariance of 
${\cal D}$ under the complex structure $J$.

The above computations are valid for any K\"ahler metric on $\Omega_\alpha$ which 
is invariant under the group $G$ and complex conjugation. In particular, it also holds true
for the Bergman metric on $\Omega_\alpha$. Let us now assume in addition that the metric
$g$ is Kähler--Einstein, with Einstein constant $-(2n+2)$. 
Then the value of the Ricci tensor of $g$, applied to a unit tangent vector in ${\cal D}$, equals
$-(2n+2)$. Inserting this value into the equation (\ref{odegeneral}) is equivalent to the 
differential equation (\ref{ode}) for the function $f=f_\alpha$ stated in the theorem. 

As a consequence, we obtain that the growth function $f=f_\alpha$ for the invariant Kähler--Einstein metric on 
$\Omega_\alpha$ is a solution of the equation (\ref{ode}).
This completes the establishment of the differential equation (\ref{ode}) for $f_\alpha$. 

We are left with showing that the initial condition for the solution $f_\alpha$ of the 
differential equation (\ref{ode}) which determines
the metric on $\Omega_\alpha$ for $\alpha \geq 1$ is a 
condition $f_\alpha(0)\in (\sqrt{\frac{n}{n+1}},1]$ and $f_\alpha^\prime(0)=0$, and that
the map $\alpha\to f_\alpha(0)$ is a decreasing homeomorphism 
$[1,\infty)\to (\sqrt{\frac{n}{n+1}},1]$.

By invariance of the metric under the reflection in the $z_1$-coordinate,
we know that
$f_\alpha^\prime(0)=0$.
%No two domains $\Omega_\alpha,\Omega_\beta$ for $\alpha\not=\beta$ can have isometric 
%invariant Kähler--Einstein metrics as by Corollary 1 of \cite{AS83}, 
%they are not biholomorphic.
%As the metric $g_\alpha$ is 
%invariant under the isometric reflection of the disk $D$ at $0$ (which corresponds to 
%the isometry $e^{i\pi}\in S^1$), we have $f_\alpha^\prime(0)=0$.  

%{\color{blue} H. Maybe I need to think a bit more about this, but the argument seems to only be saying that if %$f_\alpha(0)=f_\beta(0)$, then there exists a $C^\infty$ isometry $(\Omega_\alpha, g_\alpha)\to (\Omega_\beta, g_\beta)$. We still %need to say that it preserves the complex structure, right?} 
Observe 
that the metric $g_\alpha$ on $\Omega_\alpha$ is completely determined by the function $f_\alpha$. 
Namely, $f_\alpha(0)$ determines the restriction of $g_\alpha$ to the divisor $B_0$. 
Furthermore, let us consider a standard totally real plane $H$ containing a geodesic line $\eta$ in 
$B_0$ through $0$. This plane is foliated by geodesics orthogonal to $\eta$, parameterized by
arc length with respect to the metric $g_\alpha$. The function $f_\alpha$ completely determines the metric 
on $H$ in these coordinates as it determines the length of the tangent vectors orthogonal 
to the tangents of these geodesics. In particular, 
it computes for every $t>0$ the metric on the 
$J$-invariant subbundle ${\cal D}$ of the tangent bundle of the 
real hypersurface $G\gamma(t)$ as a multiple of the pull-back of the metric on $B_0$ under the 
natural projection.

Now viewing the disk $D$ as the $S^1$-orbit of the geodesic $\gamma$ in $H$ through $0$ which is orthogonal
to $B_0$, we know that we can also recover 
the restriction of the metric $g_\alpha$ to the disk $D$ 
by knowing the curvature of the metric and hence the growth of the lengths of the $S^1$-orbits.

%The metric $g_\alpha$ on the holomorphic disk $D$ is conformal for the standard complex structure,
%and the identification 
%of the metric on $D$ determined by the function $f_\alpha$ is in terms of a conformal factor with respect to 
%the standard metric. Thus 

As a consequence, if $\alpha\not=\beta$ but $f_\alpha(0)=f_\beta(0)$ then there exists an 
$G$-equivariant isometry 
$(\Omega_\alpha,g_\alpha)\to (\Omega_\beta,g_\beta)$ whose restriction to the disk $D$ is a biholomorphic map. 
By equivariance under the action of the group $G$ this isometry commutes with the complex structure 
and hence is a biholomorphic map. By Corollary 1 of \cite{AS83}, this is impossible. 

As a consequence, the map 
$\alpha\to f_\alpha(0)$ is injective. 
As $f_1$ defines the metric on the ball, 
to complete the 
proof of the theorem 
is suffices to show the following statement.

\smallskip\noindent
{\bf Claim 3.} The map $\alpha\mapsto f_\alpha(0)$ is continuous, and $f_\alpha(0)\to \sqrt{\frac{n}{n+1}}$ as $\alpha\to \infty$.

\smallskip\noindent 
\emph{Proof of Claim 3.}
%We also fix a compact ball $K\subset B$ centered at $0$.
%
%Let $g^B_\alpha$ be the Bergman metric on $\Omega_\alpha$. 
%We know that $g^B_1$ is the standard metric on the ball, and 
%$g^B_\infty$ is the product metric on $D\times B_0$. 
%Furthermore, for any $\alpha\leq \beta\in [1,\infty)$
%we have $\iota_{\alpha,\beta}^*g^B_\beta \leq g^B_\alpha$. 
%In particular, the restrictions of the metrics 
%$g^B_\alpha$ to a fixed compact ball $K\subset B$ centered at $0$ is 
%uniformly bounded from above and below. 
%
%\textcolor{red}{I am not sure what the best route is here. 
%Establishing that on $K$ the KE metric is uniformly comparable to the 
%Bergman metric will do the job. One could write down the Bergman metric and 
%verify that it has bounded geometry in the sense of Cheng-Yau or work on the
%fixed ball and verify this or use volume comparison where this is is not 
%immediate either}
%
Put $\Omega_\infty=D\times B_0$ where $D\subset \mathbb{C}$ is the standard unit disk.
For $1\leq \alpha \leq \beta\leq \infty$ let 
\[\iota_{\alpha,\beta}:
\Omega_\alpha\to \Omega_\beta\] 
be the natural $G$-equivariant inclusion. 

Denote as before by $g_\alpha$ the Kähler--Einstein
metric on $\Omega_\alpha$ with Einstein constant $-(2n+2)$. Let $\omega_\alpha$ be
the K\"ahler form associated to $g_\alpha$. Since $\omega_\alpha$ is K\"ahler-Einstein, one can find a potential $\phi_\alpha$ for the metric 
(that is,  $\omega_\alpha= dd^c \phi_\alpha$) such that  
\begin{equation}
\label{KE}
\omega_\alpha^n=e^{2(n+1)\phi_\alpha}\omega_{\mathbb{C}^n}^n,
\end{equation}
where $\omega_{\mathbb{C}^n}$ is the standard euclidean metric.

We next derive some uniform estimates for $\omega_\beta$ as $\beta$ ranges in $[1,+\infty]$. 
First, since $\omega_\beta$ is K\"ahler-Einstein, of negative Ricci curvature,  
Theorem 3 of \cite{Y78} shows that there is a universal constant $C>0$ so that
\begin{equation}
\label{vol form}    
\iota_{\alpha,\infty}^*\omega_\infty^n \leq C\iota_{\alpha,\beta}^*\omega_\beta^n \leq C^2\omega_\alpha^n
\end{equation}
holds on $\Omega_\alpha$ for any $\beta \geq \alpha$. 

The K\"ahler-Einstein metric $\omega_\infty$ on $D\times B_0$ 
is just the product of suitably scaled complex hyperbolic metrics on each factor 
and hence it has negative holomorphic sectional curvature. Therefore, Theorem 1 of \cite{Roy80} shows that there is a constant $c>0$ independent of $\beta$ such that 
\begin{equation}
\label{comparison 1} 
 c \, \iota_{\beta, \infty}^* \omega_\infty \leq  \omega_\beta
\end{equation}
holds for any $\beta\ge 1$.

Finally, if $\omega_B$ denotes the complex hyperbolic metric on the unit ball $B\subset\mathbb C^n$, then Theorem 2 of \cite{Y78} shows that there is a constant $c'>0$ independent of $\beta$ such that 
\begin{equation}
\label{comparison 2} 
 c' \, \Phi_{\beta}^* \omega_B \leq \omega_\beta 
\end{equation}
where $\Phi_\beta:\Omega_\beta\to B$ is the holomorphic map defined in the 
beginning of this section. Strictly speaking,
$\Phi_\beta$ is multivalued when $\beta$ is not an integer,  but $\Phi_{\beta}^* \omega_B$ is well-defined.
Therefore, pointwise computations can be done by choosing a local branch and one can then apply the maximum principle just as in \cite{Y78}.\\

As a consequence of \eqref{vol form} and \eqref{comparison 1}, the following holds true. 
Let $K \subset \Omega_\alpha$ be a compact set and let 
$\epsilon >0$ be sufficiently small that for $\vert \alpha-\beta\vert <\epsilon$, we have
$K\subset \Omega_\beta$. Then for $\vert \alpha -\beta \vert \leq \epsilon/2$, there is a constant $C_K$ independent of $\beta$ such that 
\[C_K^{-1}\omega_{\mathbb C^n}\leq\omega_\beta \leq C_K \omega_{\mathbb C^n} \quad \mbox{on } K.\]

Given the complex Monge-Amp\`ere equation \eqref{KE}, a standard bootstrapping argument yields that if $\beta_i\to \alpha$ is any convergent sequence,
then by passing to a subsequence, we may assume that the K\"ahler metrics 
$\omega_{\beta_i}$ converge uniformly on $K$ to a K\"ahler metric $\hat \omega$
on $K$. This metric then is K\"ahler-Einstein, with constant $-2(n+1)$.
As $\Phi_\alpha$ depends in an analytic fashion on $\alpha$, we have 
$\Phi_{\beta_i}^*\omega_B\vert_K\to \Phi_\alpha^*\omega_B\vert_K$ and hence 
\eqref{comparison 2} shows that $\hat\omega \geq c' \, \Phi_{\alpha}^* \omega_B$. 
As $K\subset \Omega_\alpha$ was arbitary, using a diagonal sequence we deduce
that $\hat \omega$ is defined on all of $\Omega_\alpha$.
Since $\Phi_\alpha$ is proper, this implies that $\hat \omega$ is complete. 
Theorem \ref{chengyau} then yields 
that $\hat \omega=\omega_\alpha$. In particular, the assignment $\alpha\mapsto \omega_\alpha(0)$ is continuous with respect to the usual topology on $[1,+\infty]$ and $\Lambda^{2}\mathbb R^{2n}$, respectively. 

As $f_\alpha(0)$ determines the scaling
factor of the restriction of $g_\alpha$ to $B_0$ with respect to the Kähler--Einstein metric on $B_0$ with constant $-2(n+1)$, we conclude that the map $\alpha\mapsto f_\alpha(0)$ is continuous. 
This continuity is also valid for $\alpha=1$, which corresponds to the ball, and 
for $\alpha=\infty$ which corresponds to the product $D\times B_0$.
As $f_1(0)=1$ and $f_\infty(0)=\sqrt{\frac{n}{n+1}}$ (the latter value describing
the product Kähler--Einstein metric), injectivity of the assignment 
$\alpha\mapsto f_\alpha(0)$ yields that $f_\alpha(0)\in (\sqrt{\frac{n}{n+1}},1]$ for all 
$\alpha \geq 1$. 
This completes the proof of the claim.
\hfill $\blacksquare$
\end{proof}

\begin{remark}
\label{3 and 3'}
Claim 2 in the proof of Theorem \ref{thm:ode}, which is valid for any $G$-invariant
K\"ahler metric,  
implies the equivalence of assumption (3) in Corollary \ref{negative} and 
condition ($3^\prime$) stated in its proof and hence completes the proof of Corollary \ref{negative}.
\end{remark}

\subsection{The curvature of the Kähler--Einstein metric}\label{sec:curvature}

The goal of this section is to analyze the solutions of 
the differential equation (\ref{ode}) and use it to control the  curvature 
of the Kähler--Einstein metric $g_\alpha$ on the domain
$\Omega_\alpha$ $(\alpha <\infty)$ 
with Einstein constant $-(2n+2)$. 
The following theorem summarizes the relevant curvature properties.

\begin{theorem}\label{negativeonreal}
Let $g_\alpha$ be the invariant Kähler--Einstein metric on 
the domain $\Omega_\alpha\subset \mathbb{C}^n$ 
Then the following holds true.
\begin{enumerate}
\item 
The sectional curvature of a standard totally real
plane $H\subset \Omega_\alpha$ is negative and bounded from below by $-1$. 
\item The sectional curvature $K_{\alpha}$ of the complex disk $D$ is negative and 
contained in the interval $(-2n-2,-4]$. For every $\epsilon >0$ 
there exists a number $C=C(\alpha,\epsilon)>0$
such that $\big|K_\alpha+4\big|\leq Ce^{-(1-\epsilon)\,d(0,\cdot)}$. 
\item The sectional curvature is bounded from above by a negative constant, and bounded 
from below by $-2n-2$. 
\end{enumerate}
\end{theorem}

\begin{proof}[Proof of Theorem \ref{negativeonreal}] For convenience, we drop the index $\alpha$ from the notation.
By the first part of Theorem \ref{thm:ode}, we know that the invariant Kähler--Einstein metric $g=g_\alpha$
on $\Omega_\alpha$ determines a solution $f=f_\alpha$ of the differential equation (\ref{ode}) with initial condition 
$f(0)\in (\sqrt{\frac{n}{n+1}},1]$ and
$f^\prime(0)=0$. 
It is a direct consequence of the equation that we have 
$f^{\prime\prime}(0)>0$ and hence $f^\prime(t)>0$ for $t>0$ sufficiently close to $0$.
We divide the argument
into six claims. 

\smallskip\noindent
{\bf Claim 1:} The function $\log f$ is convex, that is, $\frac{d}{dt}\frac{f^\prime}{f}=\frac{d^2}{dt^2}\log f=
\frac{f^{\prime\prime}}{f}-(\frac{f^\prime}{f})^2\geq 0$.

\smallskip \noindent
\emph{Proof of Claim 1.}
The inequality clearly holds true for $t=0$. Assume to the contrary 
that there exists a smallest number $\tau >0$ so that 
$(f^{\prime\prime}/f-({f^\prime}/{f})^2 )(\tau)=0$ and that this quantity is negative for 
$t\in (\tau,\tau+\delta)$ for some small $\delta >0$. 
This means that the value of $f^\prime/f$ is strictly decreasing on $(\tau,\sigma)$
for some $\sigma\in (\tau,\tau+\delta)$. 

Since $f^\prime/f$ is non-decreasing on $[0,\tau]$, and $f^\prime(t)>0$ for sufficiently small $t>0$, 
by possibly decreasing $\delta$ we may assume that
$f^\prime >0$ on $(\tau-\delta,\tau+\delta)$.
Then $(f^\prime)^2/f^2$ is also strictly 
decreasing on $(\tau,\sigma)$
and hence $n+1-(n+1)\frac{(f^\prime)^2}{f^2} -\frac{1}{f^2}$ is strictly increasing on 
$(\tau,\tau +\delta)$. 
Inserting into the equation (\ref{ode}) yields a contradiction which shows the claim.
\hfill$\blacksquare$

As a consequence of Claim 1, we have 
$f^{\prime\prime}(t)>0$ for all $t$. In particular it holds 
$\frac{f^{\prime\prime}}{f}(t)>0$ for all $t$. Moreover, $f^\prime$ is strictly
increasing in $t$ and hence $f^\prime>0$ on $(0,\infty)$, which yields that 
$f$ is strictly increasing on $(0,\infty)$ as well.

%\smallskip \noindent
%\emph{Proof of Claim 1.}
%We argue by contradiction and we assume that there exists a smallest number
%$\tau>0$ so that $f^{\prime\prime}(\tau)=0$. Since $f^\prime(s) >0$ for small $s>0$ and 
%$f^{\prime\prime}(s)>0$ for $s\in [0,\tau)$, 
%we have $0=f^{\prime\prime}(\tau) < f^\prime(\tau)$ and hence 
%\[\frac{d}{dt}\frac{f^\prime}{f}\vert_{t=\tau}=
%\frac{f^{\prime\prime}(\tau)}{f(\tau)}-\frac{f^\prime(\tau)^2}{f(\tau)^2}<0.\]
%
%As $f^\prime >0$ 
%on $(0,\tau]$ and $f(0)>0$, we know that $\frac{d}{dt}\frac{1}{f^2}\vert_{t=\tau}<0$. 
%Using strict monotonicity of the function $u\to u^2$ on $(0,\infty)$, we conclude that 
%for $\sigma<\tau$ close to $\tau$, the value at $\sigma$ of the 
%function on the left hand side of equation (\ref{ode}) is larger than the value at 
%$\tau$ which is impossible. 
%This contradiction completes the proof of the claim.
%\hfill$\blacksquare$
%
%\smallskip

%As a consequence of Claim 1, $f^\prime$ is strictly increasing in $t$. In particular, we have
%$f^\prime(t)>0$ for all $t>0$ and $f$ is strictly increasing. The next claim shows
%that $\log f$ is convex. 
%Thus $\log f^\prime$ is defined on $(0,\infty)$, and since
%$\log$ is strictly increasing, the function $t\to \log f^\prime(t)$ is strictly increasing on $(0,\infty)$. 
%Taking the derivative yields the following claim.

As the function $f^\prime/f$ is non-decreasing, 
we can ask for its limit as $t\to \infty$. 

\smallskip
\noindent{\bf Claim 2:} It holds $f^\prime/f\to 1$ as $t\to \infty$. 

\smallskip \noindent
\emph{Proof of Claim 2.} Inserting the inequality of Claim 1 into 
the differential equation (\ref{ode}) yields that
$f^\prime/f<1$ on $[0,\infty)$ and hence 
$\lim_{t\to \infty}(f^\prime/f)(t)=a\in (0,1]$. 
As $f^{\prime\prime}>0$, we have $f(t)\to \infty$ $(t\to \infty)$.
Thus if $a<1$ then the equation (\ref{ode}) shows that for all sufficiently large 
$t>0$ we have $\frac{f^{\prime\prime}}{f}>1+\epsilon$ for $\epsilon=n(1-a/2)>0$. 
But then for large $t$ the quantity $\frac{d}{dt}\frac{f^\prime}{f}(t)$ is bounded from below
by a universal positive constant which contradicts the fact that $f^\prime /f<1$.
\hfill$\blacksquare$

\smallskip

Let now $f_1(t)=\cosh(t)$ be the solution of the equation 
(\ref{ode}) with initial condition $f_1(0)=1$ and $f_1^\prime(0)=0$. Assume that $\alpha\not=1$, that is, 
$f(0)<1=f_1(0)$.

\smallskip\noindent
{\bf Claim 3:} We have $f(t)<f_1(t)$ and $(f^\prime/f)(t)<(f_1^\prime/f_1)(t)$ for all $t>0$.

\smallskip\noindent
\emph{Proof of Claim 3.}
Assume to the contrary that there is a first $\tau >0$ so that 
$f(\tau)=f_1(\tau)$. Since $\log $ is a monotone function and $f,f_1$ are positive, we then we have 
$\frac{d}{dt}\log f (\tau)\geq \frac{d}{dt}\log f_1(\tau)$, that is, 
$(f^\prime/f)(\tau)\geq (f_1^\prime/f_1)(\tau)$. But if equality holds then 
$f^\prime(\tau)=f_1^\prime(\tau)$ and hence the initial conditions at $\tau$ of the solutions 
$f,f_1$ of the equation (\ref{ode}) coincide. Then 
$f=f_1$ which is impossible. So we have 
$(f^\prime/f- f_1/f_1)(\tau)>0$.  

The equation (\ref{ode})  shows that $f^{\prime\prime}(\tau)<f_1^{\prime\prime}(\tau)$
and hence $\frac{d}{dt}(\frac{f^\prime}{f}- \frac{f_1^\prime}{f_1})\vert_{t=\tau}<0$. 
Thus the function $f^\prime/f-f_1^\prime/f_1$ is decreasing near $\tau$.
On the other hand, the initial conditions for $f,f_1$ at $t=0$ imply that
$f^\prime/f-f_1^\prime/f_1$ is also decreasing near $0$. As its value at $0$ equals zero
and its value at $\tau$ is positive,  the intermediate
value theorem yields that there 
is some smallest $\sigma \in (0,\tau]$ with $f^\prime/f(\sigma)=f_1^\prime/f_1(\sigma)$.
Since $f^\prime/f-f_1^\prime/f_1$ is decreasing near $\tau$, we have 
$\sigma <\tau$ and hence $f(\sigma)<f_1(\sigma)$ by the choice of $\tau$. 

Insertion of this inequality into (\ref{ode}) yields 
$(f^{\prime\prime}/f)(\sigma) <(f_1^{\prime\prime}/f_1)(\sigma)$ and hence
$f^\prime/f-f_1^\prime/f_1$ is decreasing near $\sigma$. This is a contradiction to the choice of 
$\sigma$. 
Together we obtain that $f(t)<f_1(t)$ for all $t$ and also $f^\prime/f<f_1^\prime/f_1$. 
\hfill$\blacksquare$

\smallskip\noindent
{\bf Claim 4:} The function $t\to \frac{f^{\prime\prime}}{f}(t)$ is 
strictly increasing, and $\frac{f^{\prime\prime}}{f}(t)\to 1$ as $t\to \infty$.

\smallskip\noindent
\emph{Proof of Claim 4.}
The equation (\ref{ode}) shows that
\[\frac{f^{\prime\prime}}{f}=n+1-n(\frac{f^\prime}{f})^2-\frac{n}{f^2}.\]
Differentiating this equations yields
\begin{equation}\label{derivativeofcurv}
\frac{d}{dt} (\frac{f^{\prime\prime}}{f})=
-2n(\frac{d}{dt} \frac{f^\prime}{f})(\frac{f^\prime}{f})+\frac{2n}{f^2}(\frac{f^\prime}{f}).
\end{equation}
Inserting the initial condition for $f$ shows that the right hand side of 
equation (\ref{derivativeofcurv})
vanishes at $t=0$. In view of $\frac{1}{f^2(0)}>1$, 
dividing by $\frac{f^\prime}{f}$, which is positive for all $t>0$ 
by Claim 1,  
and taking the limit as $t\searrow 0$ yields
that the right hand side of (\ref{derivativeofcurv}) is positive for 
small $t> 0$. 

We use equation (\ref{derivativeofcurv}) 
to study the critical points of $\frac {f^{\prime\prime}}{f}$. 
Let $\tau >0$ be a first positive critical point. Since $\frac{f^\prime}{f}(\tau)>0$, 
equation (\ref{derivativeofcurv}) yields
that 
\[-2n (\frac{f^{\prime\prime}}{f}(\tau)-(\frac{f^\prime}{f})^2(\tau))+\frac{2n}{f^2}(\tau)=0\]
and hence 
\[\frac{1}{f^2}(\tau)=\frac{f^{\prime\prime}}{f}(\tau)-(\frac{f^\prime}{f})^2(\tau).\]
Insertion of the expression for $\frac{1}{f^2}(\tau)$ into the differential equation 
(\ref{ode}) shows that $\frac{f^{\prime\prime}}{f}(\tau)=1$. 

Now by Claim 1 and Claim 2, we have $\frac{f^{\prime\prime}}{f}\geq (\frac{f^\prime}{f})^2$, 
moreover $\frac{f^\prime}{f}$ is increasing and converges to $1$ as $t\to \infty$. 
Using once more equation (\ref{ode}), we also have $\frac{f^{\prime\prime}}{f}\to 1$ $(t\to \infty)$. 
Thus if there exists a number $t>0$ so that $\frac{f^{\prime\prime}}{f}(t)>1$, then 
the function $\frac{f^{\prime\prime}}{f}$ assumes a global maximun at a number $t>0$ with 
$\frac{f^{\prime\prime}}{f}(t)>1$. But then $t$ is a critical point for 
$\frac{f^{\prime\prime}}{f}$ violating that by the above computation, its value at
every critical point is one. 

We conclude that $\frac{f^{\prime\prime}}{f}(t)\leq 1$ for all $t$, moreover the 
only critical points in $(0,\infty)$ are global maxima with functional value one. As $\frac{f^{\prime\prime}}{f}(t)\to 1$
$(t\to \infty)$, we deduce that 
the function $\frac{f^{\prime\prime}}{f}$ 
is non-decreasing. Since it also is analytic, it can not assume the 
value one as this would imply that the function is constant. 
Hence $\frac{f^{\prime\prime}}{f}$ 
is strictly increasing as predicted in the claim.
\hfill$\blacksquare$

%\smallskip\noindent
%{\bf Claim 6:} The function $u:=\frac{1+f'^2}{f^2}$ is decreasing, hence satisfies $ 1< u \leq \frac{1}{f(0)^2}.$
%
%\smallskip\noindent
%\emph{Proof of Claim 6.}
%First, observe that $f(t)\to \infty$ since $f$ is increasing and $f'/f\to 1$ by Claim 3. This implies that $u\to 1$ as $t\to \infty$. So we only %need to prove that $u$ is increasing. The derivative of $u$ is given by
%\[u'=\frac{2f'}{f}\Big(\frac{f''}{f}-\frac{1+f'^2}{f^2}\Big)= 2(n+1)\frac{f'}{f}\cdot (1-u).\]
%Finally, Claim 5 couple with the ODE solved by $f$ shows that $nu-(n+1)>-1$, i.e. $u>1$ and the claim follows.
%\hfill$\blacksquare$

We can now use what we established to give an explicit description of the curvature of the 
metric $g_\alpha$. To this end 
we need to control the Gaussian curvature $K_\alpha(t)$ of the totally geodesic holomorphic disk $D$, the  curvature $K_{\rm tr}(t)$ of a totally real plane and the curvature of the planes in $\mathcal D$. We have 
\[K_{\rm tr}(t)=-\frac{f^{\prime\prime}}{f}\]
by \eqref{cur1} so that Claim 4 yields
\begin{equation}
    \label{Ktr} 
    -1 \le K_{\rm tr} \le -(n+1)+\frac{n}{f(0)^2}.
\end{equation}
By Claim 2 from the proof of Theorem \ref{thm:ode}, we have 
\[R|_{\Lambda^2\mathcal D}= \frac{1}{f^2}((f^\prime)^2+1) R_0\]  
where the function $\frac{1}{f^2}((f^\prime)^2+1)=-\frac 1n \frac{f''}{f}+\frac{n+1}{n}$ is decreasing by Claim 4. 
Recall that the sectional curvature of 
the metric $g_0$ on the ball $B_0$ is contained in the interval
$[-4,-1]$ since $g_0$ has holomorphic sectional curvature $-4$. This implies that for any plane $P\subset \mathcal D$, it holds
\begin{equation}
\label{K hor}
-\frac{4}{f(0)^2}\le K(P)\le -1.
\end{equation} 
Finally, since $\Ric g=-2(n+1)g$, we have that 
\begin{equation}
    \label{Kalpha}
    K_\alpha=-2(n+1)-2(n-1)K_{\rm tr},
\end{equation}
hence
\begin{equation}
    \label{KD}
    -2(n+1)\le K_\alpha\le -4.
\end{equation}

By Lemma~\ref{negative}, it follows that the curvature of $g$ is negative, and the first three items of the theorem are proved. Moreover, 
Claim 4 implies that $K_{\rm tr}(t)\to -1$ and $K_\alpha(t)\to -4$ as $t\to \infty$. Finally, we see that the supremum of the sectional curvature is attained by the totally real planes at a point of $B_0$. That is, 
\begin{equation}
    \label{sup K}
    \sup_{z\in \Omega_\alpha}\sup_{\substack{P\subset  T_z\Omega_\alpha\\ \mathrm{plane}}} K_{g_\alpha}(P)= -(n+1)+\frac{n}{f_\alpha(0)^2}. 
\end{equation}
Note that as $\alpha\to +\infty$, the right hand side increases to $0$.\\
 
The following computation yields that the convergence of the curvature tensor
to the curvature tensor of a metric on the ball is exponential in $t$.

\smallskip\noindent
{\bf Claim 5:} For $\epsilon >0$ there exists a number 
$C=C(f,\epsilon)>0$ such that
\[ \left| \frac{f^\prime(t)}{f(t)}-1\right|+\left| \frac{f^{\prime\prime}(t)}{f(t)}-1\right| \leq Ce^{-(1-\epsilon)t}.\]
Moreover, for any integer $k\ge 1$, there is a constant $C_k=C(k,f,\epsilon)$ such that 
\begin{equation}
    \label{derivatives estimates}
\Big|\frac{d^k}{dt^k}\Big(\frac{f''}{f}\Big)\Big| \le C_k e^{-(1-\epsilon)t}.
\end{equation}

\smallskip\noindent
\emph{Proof of Claim 5.}
We know that $\frac{f^\prime}{f}\leq \frac{f^{\prime\prime}}{f} \leq 1$
for all $t$. On the other hand, we also have
$\frac{f^{\prime\prime}(t)}{f(t)}+ n\frac{(f^\prime(t))^2}{f(t)^2}\geq n+1-n\frac{1}{f^2}$.
For large $t$, $\frac{d}{dt}\log f(t)\geq 1-\epsilon$
and hence  $n/f^2(t)\leq e^{-(1-\epsilon)t}$. 
From this we get $\frac{f^\prime(t)}{f(t)}-1\ge Ce^{-(1-\epsilon)t}$. The remaining inequality for $\frac{f^{\prime\prime}(t)}{f(t)}-1$ is a consequence of the previous one and \eqref{ode}.  

As for the last statement of the claim, let us set $u=\log f$. We have $u''=\frac{f''}{f}-(\frac{f'}{f})^2= \frac{f''}{f}-u'^2$ and $u''+(n+1)(u'^2-1)+ne^{-2u}=0$ thanks to \eqref{ode}. From the previous steps, we know that $e^{-u}+|u'-1|+u''= O(e^{-(1-\epsilon)t})$. An elementary induction based on the ODE for $u$ and the previous estimates shows that $u^{(k)}$ converges to zero exponentially fast for any $k\ge 2$. Since $u''= \frac{f''}{f}-u'^2$, the same goes for the derivatives of $\frac{f''}{f}$, hence the claim. 
\hfill$\blacksquare$

From Claim 5, \eqref{cur1} and \eqref{Kalpha}, one deduces the second item in the Theorem, which concludes its proof.
\end{proof}

\begin{remark}
J.F. Lafont and B. Minemyer \cite{LM25} made independent computations to analyze 
(real) Einstein metrics on $\Omega_\alpha$. Combined with our results, their work leads to 
an explicit solution of the differential equation (\ref{ode}) with respect to the 
initial conditions $f(0)\in (\sqrt{\frac{n}{n+1}},1], f^\prime(0)=0$.  
\end{remark}

\subsection{Comparison with the pull-back of the ball metric}

In Theorem \ref{negativeonreal}, we established a precise curvature control for the 
Kähler--Einstein metric $g_\alpha$ on $\Omega_\alpha$. In particular, it 
follows from its second part that the curvature tensor of $g_\alpha$ converges 
exponentially with the distance from the divisor $B_0$ to the curvature tensor of 
a metric of constant holomorphic sectional curvature and the same Einstein constant. 
In this section it will be convenient to normalize this constant to be $\frac{1}{2}(n+1)$ so that 
the holomorphic sectional curvature for the metric on the ball equals $-1$.

The pull-back $\Phi_\alpha^* g_1$ by $\Phi_\alpha$ of the metric $g_1$ on the ball $B$ 
is a metric of constant holomorphic sectional curvature $-1$ on $\Omega_\alpha \setminus B_0$. 
The goal of this section is to compare the metrics $g_\alpha$ and 
$\Phi_\alpha^*g_1$ as the distance $d_\alpha(B_0,\cdot)$ from 
$B_0$, measured with respect to the distance function $d_\alpha$ of $g_\alpha$, 
tends to infinity. Our findings are summarized in the following result. 

\begin{theorem}\label{thm:comparison}
For $k\geq 0$ 
there exist numbers $a(\alpha,k)>0,C(\alpha,k)>0$ 
such that 
the metrics $g_\alpha$ and $\Phi_\alpha^*g_1$ satisfy 
$\Vert g_\alpha-\Phi_\alpha^*g_1\Vert_{C^k} \leq C(\alpha,k)e^{-a(\alpha,k) \,d_\alpha(\cdot,B_0)}$ on 
the complement of the tubular neighborhood of radius one about $B_0$. 
\end{theorem}

\begin{remark}\label{reformulate}
Since the map $\Phi_\alpha$ is singular on $B_0$, the pull-back $\Phi_\alpha^*g_1$ is not a metric on 
$\Omega_\alpha$, but it is a Kähler--Einstein metric on $\Omega_\alpha -B_0$. 
Theorem \ref{thm:comparison} says that 
this metric is arbitrarily close to the metric 
$g_\alpha$ in the $C^k$-topology on the complement of a suitable tubular neighborhood of $B_0$ in 
$\Omega_\alpha$, measured  with respect to the metric $g_\alpha$. 
Since any such tubular neighborhood is the preimage under $\Phi_\alpha$ 
of a tubular neighborhood of $B_0$ in $B$ for the complex hyperbolic metric $g_1$, we can rephrase 
the result also in terms of the distance from $B_0$ with respect to 
the pull-back $\Phi_\alpha^* g_1$ provided that we restrict measuring distances 
to the complement of the preimage of the radius one tubular neighborhood of $B_0$ in $B$.
\end{remark}

\begin{proof}[Proof of Theorem \ref{thm:comparison}]
Let $\omega$ (resp. $\omega_1$) 
be the K\"ahler form associated to $g_\alpha$ (resp. $g_1$). 
Put $\widehat \omega=\Phi_\alpha^*\omega_1$. The two-form $\widehat \omega$ on $\Omega_\alpha-B_0$ defines a Kähler--Einstein metric 
of constant holomorphic sectional curvature $-1$. 
The punctured holomorphic disk $D\setminus \{0\}\subset D=\{z_i=0 \text{ for } i\geq 2\}\subset \Omega_\alpha$ is totally geodesic 
for both $\omega, \widehat \omega$. 
By invariance of $\omega$ under $G$ (resp. $\omega_1$ under $\mathrm{PU}(n, 1)$) and equivariance of the map $\Phi_\alpha$ as formulated in \eqref{equivariance}, 
the two-forms $\widehat \omega^D=\widehat \omega\vert_D$ and $\omega^D=\omega\vert_D$ 
are invariant under the circle group $S^1$ of rotations acting on $D$.

We begin with showing that the metrics $\omega^D,\widehat \omega^D$ are exponentially close with the 
distance from $0\in D$, where closeness means pointwise closeness in 
norm with respect to the metric $\omega^D$,
equivalently $C^0$-closeness.
The proof of this statement is carried out in three steps. Throughout we denote for $r>0$ by 
$D_r\subset D$ the disk of radius $r$ about $0$ for the metric $\omega^D$.

\smallskip\noindent
{\bf Claim 1.} There exists a number $\kappa=\kappa(\alpha)>0$ so that 
$\frac{\widehat \omega^D}{\omega^D}\in [\kappa,\kappa^{-1}]$ on $D\setminus D_1$.

\smallskip\noindent
\emph{Proof of Claim 1.}
Choose a function $\phi$ supported in $D_1$ 
such that 
$\widehat \omega^D+dd^c\phi$  
is a K\"ahler metric on $D$ of bounded 
negative curvature. The existence of such a function is standard, see e.g. \cite{Zh96}, 
its proof will be omitted. 
Since the curvature of $\omega^D$ is also bounded negative, 
the classical Schwarz Pick lemma (see \cite{Y78} for more information) 
shows that $\frac{\widehat \omega^D+dd^c\phi}{\omega^D}\in [\kappa,\kappa^{-1}]$
for some constant $\kappa=\kappa(\alpha) >0$. 
In particular, we have $\frac{\widehat\omega^D}{\omega^D}\in [\kappa,\kappa^{-1}]$ on $D-D_1$.
\hfill$\blacksquare$

\smallskip\noindent

Let $d=d_\alpha$ be the distance function on $D$ for the metric $\omega^D$. 
To simplify the notations, 
by modifying $\widehat \omega^D$ with a potential supported in $D_1$ 
we assume that $\widehat \omega^D$
is a complete K\"ahler metric on $D$. As we are only interested in estimates outside of 
$D_1$ this does not alter our analysis.

\smallskip\noindent
{\bf Claim 2.} There exist numbers $a_1>0,C_1>0$ so that 
\[\omega^D<(1-C_1e^{-a_1\,d(0,\cdot)})^{-1}\widehat \omega^D.\]

\smallskip\noindent
\emph{Proof of Claim 2.}
By Claim 1, distances from $0$ in $D$ with respect to the distance functions of 
$\omega^D$ and $\widehat \omega^D$ are uniformly comparable. This implies that
there exists a number $\kappa \in (0,1)$ such that 
for every
$x\in D\setminus D_1$, we have $\kappa^{-1}d(0,x)>\widehat d(0,x)>\kappa d(0,x)$. 
Here $\widehat d$ is the distance function of the metric $\widehat \omega^D$.

Assume that $d(x,0)>2 \kappa^{-1}$ and 
let $\widehat Q_x\subset D\setminus \{0\}$
be the ball of radius $u=\frac{\kappa}{2}d(0,x)\leq \frac{1}{2}\widehat d(0,x)$ about $x$ for 
the metric $\widehat \omega$. Since $\widehat d(x,0)>2$, the curvature of the 
restriction of the metric $\widehat \omega^D$  
to $\widehat Q_x$ is constant $-1$.
Then up to isometry, the set $\widehat Q_x$ is the round disk in $\mathbb{C}$ of 
euclidean radius $r\in (0,1)$, with $x$ corresponding to 
the center $0$ of the disk, and equipped with the restriction of the Poincar\'e metric 
$\frac{4 dz\wedge d\bar z}{(1-\vert z\vert^2)^2}$. The Euclidean radius 
$r>0$ of $\widehat Q_x$ is computed by the formula $\cosh(u)=\frac{1+r^2}{1-r^2}$. 

Denote by $\widehat \omega_x$ the standard \emph{complete} Poincar\'e metric on $\widehat Q_x$, obtained 
as the pull-back of the 
Poincar\'e metric $\frac{4 dz\wedge d\bar z}{(1-\vert z\vert^2)^2}$ on the unit disk
by the scaling map $z\to \frac{1}{r}z$. 
This rescaling operation 
replaces the restriction of the Poincar\'e metric $\frac{4 dz\wedge d\bar z}{(1-\vert z\vert^2)^2}$
to $\widehat Q_x$ by the metric $\frac{4dz\wedge d\bar z}{(1-r^{-2}\vert z\vert)^2}$
 A standard calculation shows that 
\begin{equation}\label{apriori}
    \widehat \omega^D (x)\leq \widehat\omega_x(x)\leq 
(1-Ce^{-\kappa\, d(0,x)/2})^{-1}\widehat \omega^D(x)
\end{equation} 
for a universal constant $C>0$.

Write $\widehat \omega_x=e^{2\rho_x}\omega^D\vert \widehat Q_x$ for a function $\rho_x$ on $\widehat Q_x$.
Since $\widehat \omega_x$ is a complete metric on $\widehat Q_x$ 
and by Claim 1, $\omega^D$ is bi-Lipschitz equivalent to $\widehat \omega^D$, it follows from the construction 
of $\widehat \omega_x$ that 
$\rho_x$ is a proper function. 

Let $K_g$ be the Gauss curvature of $\omega^D$. The curvature of 
$\widehat \omega_x=e^{2\rho_x}\omega^D$ is constant $-1$ and hence denoting by 
$\Delta_x={\rm tr}\nabla^2$ the Laplacian of 
$\widehat \omega_x$, we have
\begin{equation}\label{gausscurvature1}
K_g=e^{2\rho_x}(-1-\Delta_x(-\rho_x)).\end{equation}
Now $\rho_x$ is proper and hence it assumes a minimum at some point $y\in \widehat Q_x$. Then 
it holds $\Delta_x(\rho_x)(y)\geq 0$. On the other hand, by 
Theorem \ref{negativeonreal}, the Gauss curvature $K_g$ of 
$\omega^D$ satisfies $K_g<-1$. Insertion into the equation (\ref{gausscurvature1}) implies that 
we have
$e^{2\rho_x}(y)\geq 1$. 
As $\rho_x$ assumes a minimum at $y$, this then implies  
that $\rho_x \geq 0$ and hence $\omega^D\leq \widehat \omega_x$. 
The claim 
now follows from the estimate (\ref{apriori}).
%
%The lower estimate in the formula (\ref{controlondisk}) follows from reversing the roles of 
%$\omega>D$ and $\hat \omega^D$. 
%Namely, 
%note that by the second part of Theorem \ref{negativeonreal}, 
%there exist numbers $C>0,\sigma >0$ so that 
%the Gauss curvature $K_g$ of $\omega^D$ satisfies 
%$K_g(z)\in [-1-Ce^{-2\sigma \, d(0,z)},-1]$ for all $z$.
%Denoting by $\Delta_g$ the Laplacian of the metric
%$\omega^D$, we can compute the curvature $-1$ of the metric $\hat \omega_z$ by 
%\[-1=e^{-2\rho_z}(K_g-\Delta_g(\rho_z)).\]
%Evaluating at the minimum $y$ for $\rho_z$ where $\Delta_g(\rho_z)(y)\geq 0$, 
%we obtain $e^{-2\rho_z}(y)$
\hfill$\blacksquare$

\smallskip\noindent
{\bf Claim 3.}
There exist numbers $a_2>0,C_2>0$ so that 
\[\omega^D>(1-C_2e^{-a_2 \,d(0,\cdot)})\widehat \omega^D.\]

\smallskip\noindent
\emph{Proof of Claim 3.}
The proof of the claim follows from reversing the roles of $\omega^D$ and $\widehat \omega^D$
in the proof of Claim 2. 
Let $x\in D$ be such that $d(0,x)>2\kappa^{-1}$ and let $Q_x$ be
the metric disk of radius $\frac{1}{2}d(0,x)$ about $x$ for the metric $\omega^D$. 
By the second part of Theorem \ref{negativeonreal} 
and the triangle inequality, we have 
$K_g(Q_x)\subset [-1-Ce^{- \sigma \, d(0,x)},-1]$ for some constants $C=C(\alpha)>0$,
$\sigma=\sigma(\alpha)>0$.

Let $\widehat Q_x$ be the ball of radius $u=\frac{1}{2}\kappa d(0,x)$ about $x$ for the metric 
$\widehat \omega$. We know that $\widehat Q_x\subset \Omega_x$. 
Moreover, up to isometry, $\widehat Q_x$ is the round disk in $\mathbb{C}$ 
centered at $0$, of euclidean radius $r\in (0,1)$, and  
equipped with the metric $\frac{4 dz\wedge d\bar z}{(1-\vert z\vert^2)^2}$.
Here $x$ corresponds to the center $0$ of the euclidean disk, and 
the radius $r$ is computed by $\cosh(u)=\frac{1+r^2}{1-r^2}$.

Similar to the construction in the proof of Claim 2, replace the restriction of the metric 
$\widehat \omega^D$ to $\widehat Q_x$ by an incomplete conformal metric $\widehat \omega_x$ which 
is the pull-back of the Poincar\'e metric on the unit disk by a scaling map $z\to sz$. 
Here the scaling parameter $s<1$ 
is chosen in such a way that the pull-back metric $\widehat \omega_x$ can be written as 
$\widehat \omega_x=e^{2\psi}\widehat \omega^D$ where the function $\psi$ satisfies 
$e^{2\psi} \equiv \kappa^2/4$ on 
$\partial \widehat Q_z$. 
Explicitly, the parameter $s$ is determined by the equation
$(1-s^2r^2)^2=\frac{\kappa^2}{2(1-r^2)^2}$. 
As in the proof of Claim 2, we have the estimate
\begin{equation}\label{psi}
e^{2\psi(x)}>1-Ce^{-\kappa \,d(z,0)/2}\end{equation}
for a universal constant $C>0$. 

Assume from now on that $d(x,0)$ is sufficiently large that $e^{2\psi(x)}\geq \frac{1}{2}$.
There exists a smooth 
function $\rho_x$ on $\widehat Q_x$ such that 
$\widehat \omega_x= e^{2\rho_z}\omega^D\vert \widehat Q_x$. Note that by construction, we have
$e^{2\rho_x}=e^{2\psi}\frac{\widehat \omega^D}{\omega^D}$. 
As $\widehat \omega^D\leq \kappa^{-1}\omega^D$ and
$e^{2\psi}\equiv \frac{\kappa^2}{4}$ on $\partial \widehat Q_x$, 
the value of the function $e^{2\rho_x}$ on $\partial \widehat Q_x$ is smaller than
$\kappa/4$. Moreover by the estimate (\ref{psi}) and the assumption on $d(0,x)$,  
we have $e^{2\rho_x}(x)\geq 
\kappa/2$. 
Thus $\rho_z$ 
assumes a maximum at an interior  point $y\in \widehat Q_x$. Denoting by 
$\Delta_g$ the Laplacian for the metric $\omega^D$, it follows
$\Delta_g(\rho_x)(y)\leq 0$. 

Now the constant curvature $-1$ of the metric $\widehat \omega_x$ can be computed by 
\[-1=e^{-2\rho_x}(K_g-\Delta_g(\rho_x)).\]
Since $K_g(y)\geq -1-Ce^{-\sigma\, d(0,x)}$, we obtain 
$e^{-2\rho_x}(y)\geq (1+Ce^{-\sigma\, d(0,x)})^{-1}$ and hence 
$e^{2\rho_z}(y)\leq 1+Ce^{-\sigma\, d(0,x)}$. Since $y$ was a maximum for $\rho_x$, we also have
$e^{2\rho_x}(x)\leq 1+Ce^{-\sigma\, d(0,x)}$. Together with the
estimate (\ref{psi}), this completes the proof of the claim.
\hfill$\blacksquare$

\begin{remark}
    There is an alternative, slightly different way to prove Claims 2 and 3 above, which we briefly sketch now. Write $\omega^D=e^f \widehat \omega^D$ and set $f(z)=g(t)$ where $t=\ \log |z|^{2\alpha}$. Using the curvature decay of $K_\alpha$ and Claim 1, we see that $g$ satisfies the double sided inequality $1+C(-t)^{\gamma} \ge e^{-g}\Big(\alpha^2e^{-t}(1-e^t)^2g''(t)+1\Big)>1$ for some $C>0$ and $\gamma\in(0,1)$. Using the maximum principle, one can prove that $g(t)\to 0$ as $t\to 0^-$. Moreover, the inequality above implies that $g(t)+ C(-t)^{\gamma}$ is concave, equal to $+\infty$ (resp. $0$) at $t=-\infty$ (resp. $t=0$), hence it is non-negative. Similarly, $g(t)+t-C(-t)^{\gamma}$ is convex, equal to $-\infty$ (resp. $0$) at $t=-\infty$ (resp. $t=0$), hence it is nonpositive. This yields the desired estimate $|g(t)|\le C(-t)^{\gamma}$ near $t=0$.
\end{remark}

\smallskip\noindent
{\bf Claim 4.}
Put $\rho=\frac{1}{2} \log \frac{\omega^D}{\hat \omega^D}$. For any integer $k\ge 0$ there
are numbers $a_k>0,C_k>0$ such that $\Vert \rho \Vert_{C^{k}}<C_k e^{- a_k d(0,\cdot)}$ 
where the $C^k$-norm at a point $x$ is taken as the $C^k$-norm in standard coordinates on the
unit ball about $x$ for the metric $\hat \omega^D$.

\smallskip\noindent
\emph{Proof of Claim 4.}
By Claim 2 and Claim 3, we have
\[(1-C_1e^  {-a_1 d(0,\cdot)})\omega^D<\hat \omega^D<(1-C_2e^{-a_2 d(0,\cdot)})^{-1}\omega^D\]
and hence $\vert \rho\vert 
\leq Ce^{-a d(0,\cdot)}$ for some $C>0,a>0$ by smoothness of the 
logarithm near $1$. Equation (\ref{gausscurvature1}) together with the estimates for the 
Gauss curvature for $\omega_D$ then shows that up to increasing $C$ and decreasing $a$, we have 
\[ \vert \Delta(\rho)\vert =\vert e^{-2\rho}K_g+1\vert \leq Ce^{-a d(0,\cdot)}\]
where $\Delta$ is the Laplacian of the constant curvature metric $\hat \omega$. Standard Schauder estimates now imply that 
\[\Vert \rho \Vert_{C^{1,\alpha}}\leq C e^{-a d(0,\cdot)},\]
for a possibly different constant $C$. Since $\omega^D=e^{2\rho}\hat \omega^D$, this implies that the $C^1$ norms with respect to covariant derivatives of $\omega^D$ and $\hat \omega^D$ are uniformly comparable. Now, by \eqref{derivatives estimates} and \eqref{Kalpha}, the quantity $K_g+1$ goes to zero exponentially fast along with all its covariant derivatives with respect to $\omega^D$. Since  $\Delta(\rho) = e^{-2\rho}K_g+1$, we get exponential decay for the $C^{2,\alpha}$ norm of $\rho$ with respect to $\hat \omega^D$. We can now iterate the argument by the usual bootstrapping process.  
\hfill $\blacksquare$

\smallskip

\smallskip\noindent
We can now conclude the proof of the theorem. Set $\widehat \omega:= \Phi_\alpha^*\omega_1$ which is a Kähler metric away from $B_0$. On that locus, one can write
%{\bf Claim 5.}  For any integer $k\ge 0$, there are numbers $a_k>0,C_k>0$ such that 
%$\| \log (\frac{\omega^n}{\widehat \omega^n})\|_{C^{k}} \leq C_ke^{-a_k d(B_0,\cdot)}$, where the norm is taken with respect to the pullback metric $\widehat \omega::= \Phi_\alpha^*\omega_1$ away from $B_0$. 
%\smallskip\noindent
%\emph{Proof of Claim 5.}
\[\omega=\widehat \omega+dd^c \phi=\widehat \omega +\frac{i}{2} \partial \bar \partial \phi\]
where $\phi:=\frac{1}{2n+2}\log \Big(\frac{\omega^n}{\widehat \omega^n}\Big)$. 
By invariance of $\omega$ under $G$ (resp. $\omega_1$ under $\mathrm{PU}(n, 1)$) and equivariance of the map $\Phi_\alpha$ in the sense of \eqref{equivariance}, the function $\phi$ is $G$-invariant and hence it is determined by its restriction to the disk $D$. 

%In standard complex coordinates $(z_1,\dots,z_n)$ 
%on $\Omega_\alpha$, we have
%\[\frac{i}{2}{\partial}\bar \partial \phi =\frac{i}{2}\sum_{\ell,j} 
%(\frac{\partial^2}{\partial z_\ell\partial \bar z_j}\phi) dz_\ell\wedge d\bar z_j.\]
%In particular, it holds
%\[\omega^D= \widehat \omega^D+
%\frac{i}{2} (\frac{\partial^2}{\partial z_1\partial \bar z_1}\phi)dz_1\wedge d\bar z_1\vert_D.\]
The restriction $\phi^D=\phi\vert_D$ 
of the potential $\phi$ to the disk $D$ satisfies
\[\omega^D=\hat \omega^D+dd^c \phi^D.\]
By Claim 4, we know that for any integer $k\ge 0$ the estimate
\begin{equation}
    \label{estimate disk directions}
\big| dd^c \phi^D\big|_{C^{k}} 
\leq C_k e^{-a_k \,d(0,\cdot)}
\end{equation}
holds for some $a_k>0,C_k>0$. 
%Standard potential theory for the hyperbolic disk then yields that 
%$\vert \phi^D\vert \leq Ce^{-a d(0,\cdot)}$. 
%Using once more Schauder theory, we have $\Vert \phi^D\Vert_{C^{2,\alpha}} \leq Ce^{-a d(0,\cdot)}$. 
%By invariance under the action of 
%${\rm U}(n-1,1)$, 
%this estimate implies that 
%$\vert \phi\vert \leq Ce^{-a\, d(B_0,\cdot)}$ on all of $\Omega_\alpha$. 

But $\phi^D=\phi|_D$ and the restriction of $\phi$ to any $G$-orbit is constant. Such an orbit is a distance hyperplane 
from the totally geodesic subspace $B_0$ and hence 
at distance uniformly bounded away from $B_0$, 
it is a real 
hyperplane whose second fundamental form is uniformly bounded. 
Thus near any point $x$ of sufficiently large distance from $B_0$,
there are real coordinates $(x_1,\dots, x_{2n})$ of uniformly 
bounded derivatives of all order so that $x_1=d(\cdot,B_0)$ and
that the function $\phi$ only depends on $x_1$. Furthermore, 
if $x\in D$ then 
these coordinates can be chosen so that the disk 
$D$ is the set $x_3=\cdots =x_{2n}=0$, and if $x\not\in D$ then
these coordinates are obtained from coordinates about a point 
$\hat x\in D$ by precomposing with an element of 
$PU(n-1,1)$. 
Hence the $C^{k}$-norm of $dd^c\phi$ at a point $x\in \Omega_\alpha$ with respect to $\widehat \omega$ is bounded from above by a constant multiple of the $C^{k}$-norm of $dd^c \phi^D$ at a point $x'\in D$ with respect to $\hat \omega^D$, where $x'$ satisfies with $d(0,x')=d(B_0, x)$. The theorem now follows from \eqref{estimate disk directions}.
\end{proof}

\section{The construction of Stover and Toledo}
\label{sec ST}

Consider a compact arithmetic complex hyperbolic orbifold $\hat M$ of simplest type and dimension $n\geq 2$. Such an orbifold is given 
as follows, see e.g. \cite[\textsection~\!2.2]{McR07} or \cite{St07}. There is a purely imaginary quadratic
extension $L$ of a totally real number field $F$ with 
$[F:\mathbb{Q}]\geq 2$, and there is a Hermitian quadratic form 
\[q(z,\bar z)=-a_0z_0\bar z_0 +\sum_{i=1}^n a_i z_i\bar z_i\]
with coefficients in the ring $O_F$ of integers in $F$, anisotropic over $\mathbb{Q}$,  
so that $\hat M=\hat \Gamma \backslash \mathbb{C}\mathbb{H}^n$ where
$\hat \Gamma={\rm GL}(n+1,\mathcal O_L)\cap U(q)$. The latter sits as a lattice inside $\mathrm{PU}(n,1)$, and its action on $\mathbb{C}\mathbb{H}^n$ has finite stabilizers.
 
The involution 
\[\iota:(z_0,z_1,\dots,z_n)\to (z_0,\dots,z_{n-1},-z_n)\] preserves $q$ and hence lies in
$\hat \Gamma$. Since $\hat \Gamma$ is residually finite, there exists a finite index subgroup $\hat \Gamma' \vartriangleleft \hat \Gamma$ such that $\iota \notin \hat \Gamma' $. Since $\hat \Gamma'$ is normal, it is preserved by the action of $\iota$ by conjugation. By construction, $\iota$ descends to a non-trivial holomorphic involution of $\hat M'=\hat \Gamma' \backslash \mathbb{C}\mathbb{H}^n $. 
Its fixed point set is a totally geodesic suborbifold $\hat H'\subset \hat M'$ of codimension one which is a compact arithmetic
complex hyperbolic orbifold of simplest type in its own right.  
In general, this orbifold
is neither embedded in $\hat M'$ nor connected. As a consequence of \eqref{nested} below, one can actually choose  $\hat \Gamma'$ to be one of the congruence subgroups $\hat \Gamma_\ell$ defined in the next paragraph.

An ideal 
$\mathcal I_L$ in the ring $\mathcal O_L$ of integers of $L$ determines a 
\emph{congruence subgroup} $\Gamma$ of $\hat \Gamma$, 
defined as  the kernel of the homomorphism 
$\hat \Gamma={\rm GL}(n+1,\mathcal O_L)\cap U(q)\to {\rm GL}(n+1,\mathcal O_L/\mathcal I_L)$ induced by 
the quotient projection $\mathcal O_L\to \mathcal O_L/\mathcal I_L$. 
By construction, $\Gamma$ also is invariant 
under the involution $\iota$. By choosing a \emph{sufficiently deep}
congruence subgroup we may assume that $M=\Gamma\backslash \mathbb{C}\mathbb{H}^n$ is a manifold,
called a \emph{congruence manifold} in the sequel. By perhaps passing
to a further congruence cover we may also assume that the preimage $H$ 
of $\hat H$ is a compact embedded totally geodesic submanifold of $M$. 

Using the pair $(M,H)$ as a starting point, called a \emph{standard congruence pair} in the sequel, 
the goal of this section is to establish the following improvement of one of the main results
of \cite{ST22}. For its formulation, recall that 
the \emph{normal injectivity radius} of a totally geodesic hypersurface 
$H\subset M$ is the supremum of all numbers $R>0$ so that the tubular neighborhood of 
radius $R$ about $H$ is diffeomorphic to a disk bundle over $H$.

 \begin{theorem}\label{mainapp}
 Let $(M,H)$ be a standard congruence pair. Then for all $d\geq 2,R>0$ there exists a congruence cover 
 $\Pi:M^\prime\to M$ such that for any component $H^\prime\subset M'$ of the preimage of $H$, the following 
 holds true.
 \begin{enumerate}
 \item The normal injectivity radius of $H^\prime$ is at least $R$.
 \item There exists a finite étale cover $\Pi':N^\prime\to M^\prime$ and a degree $d$ cyclic cover 
$N\to N^\prime$, totally branched along $\Pi'^{-1}(H^\prime)$. 
\end{enumerate}
 \end{theorem}
 
% \begin{remark}
% \label{reformulation}
% In other words, we have a tower of finite covers 
% \[N\overset{p}{\to} N'\overset{\Pi^\prime}{\to} M^\prime\]
%  such that $\Pi^\prime$ is étale, and $p$ ramifies at order exactly $d$ along  $(\Pi^\prime)^{-1}(H')\subset N'$ and is étale elsewhere. Moreover, $H'$ has normal injectivity radius at least $R$. 
% \end{remark}

  \begin{remark}
 \label{reformulation2}
 In other words, if we denote by $p$ the composition $N\to N'\to M'$, we get a tower of finite covers 
 \[N\overset{p}{\to} M'\overset{\Pi}{\to} M\]
  such that 
  \begin{enumerate}
  %[label=$\bullet$]
  \item $\Pi$ is étale.
  \item $p$ is a branched cover which ramifies at order $d$ over the chosen connected component $H'\subset M'$ of $\Pi^{-1}(H)$, i.e. $p^*H'=d \,p^{-1}(H')$ as Cartier divisors. Moreover, $p$ is étale over $M'\setminus H'$. The degree of $p$ may be very large (in a non-explicit way) and $p^{-1}(H')$ may be disconnected.

   \item $H'$ has normal injectivity radius at least $R$. 
   \end{enumerate}
 \end{remark}

The proof of Theorem \ref{mainapp} proceeds in two steps. In a first step we use the results of 
\cite{Be00} to verify that 
for a given standard congruence pair $(M,H)$ and a number $R>0$ 
we can find a congruence cover $M^\prime\to M$ with the property that the normal injectivity
radius of one (and hence any) component $H^\prime$ of the preimage of $H$ is larger than $R$. 
In a second step we then invoke the main result of \cite{ST22} to complete the proof.

The ideals $\mathcal I_L\subset \mathcal O_L$ can be equipped with a norm $\vert \mathcal I_L\vert$ which is just the number of 
elements in the quotient ring $\mathcal O_L/\mathcal I_L$. For any $i\geq 1$ there are only finitely many ideals of norm
at most $i$. Each of these ideals can be factorized into prime ideals. 
We denote by $\hat \Gamma_\ell$ the congruence subgroup defined by $\mathcal O_L/I_\ell(L)$ 
where $I_\ell(L)$ is the product of all prime ideals of norm less than $\ell$. This is 
a normal subgroup of $\hat \Gamma$ of finite index, and 
the groups $\hat \Gamma_\ell$ are nested and exhaustive, which means that 
\begin{equation}
\label{nested}
\forall \ell, \quad \hat \Gamma_\ell \vartriangleright \hat \Gamma_{\ell+1}, \quad \mbox{and} \quad \bigcap_{\ell} \hat \Gamma_\ell={\rm Id}.
\end{equation} 
In other words, these groups form a \emph{congruence tower}, and for any congruence
subgroup $\Gamma$ of $\hat \Gamma$, there exists some $\ell$ so that
$\hat \Gamma_\ell \vartriangleleft \Gamma$. 
The fundamental group of the congruence cover in Theorem \ref{mainapp} will be 
a congruence subgroup from the 
fixed tower, but many other choices will do as well.

In slight deviation from the previous notations, let $H$ be 
 a compact \emph{connected} embedded totally geodesic hypersurface in the congruence manifold $M$.
The preimage of $H$ in the universal covering $\mathbb{C}\mathbb H^n$ of $M$
consists of a countable collection ${\mathcal H}$
of pairwise disjoint totally geodesic hyperplanes which are transitively permuted by the deck group 
$\Gamma=\pi_1(M)$ of $M$. 
The following is well known.

\begin{lemma}\label{malnormal}
The stabilizer ${\rm Stab}(H_0):={\rm Stab}_\Gamma(H_0)$ 
of a component $H_0\in {\mathcal H}$ is a malnormal subgroup of 
$\Gamma$. 
\end{lemma}
\begin{proof} 
If $\phi\in \Gamma\setminus {\rm Stab}(H_0)$ then $\phi(H_0)$ is a component of 
${\cal H}$ which is stabilized by $\phi {\rm Stab}(H_0)\phi^{-1}$. As $\phi$ does not stabilize 
$H_0$, this component is distinct and hence disjoint from $H_0$. If $\psi\in {\rm Stab}(H_0)\cap 
\phi {\rm Stab}(H_0) \phi^{-1}$ then $\psi$ also preserves the unique shortest geodesic arc connecting 
$H_0$ to $\phi(H_0)$ and hence it fixes it midpoint. Since $\Gamma$ acts freely on $\mathbb{C}\mathbb{H}^n$, 
this implies that $\psi={\rm Id}$, from which the lemma is immediate.
\end{proof}

As a consequence, the components of ${\cal H}$ are in bijection with the cosets
$\Gamma/{\rm Stab}(H_0)$. 

%Let now $\Gamma^\prime$ be a congruence subgroup of $\Gamma$. 
%Then 
%the stabilizer of $H_0$ in $\Gamma^\prime $ is a congruence subgroup of ${\rm Stab}(H_0)$
%which is just ${\rm Stab}(H_0)\cap \Gamma^\prime$.  
%Recall that such a congruence subgroup 
%is defined as a kernel of a map of the form 
%${\rm GL}(n+1,\mathcal O_L)\cap U(q)\to {\rm GL}(n+1,\mathcal O_L/\mathcal I_L)$ where $\mathcal I_L$ is an ideal in 
%the ring of integers $\mathcal O_L$ of $L$. 
%

Each of the groups 
$\Gamma_i=\hat \Gamma_i\cap \Gamma\vartriangleleft\Gamma$ acts on ${\mathcal H}$ as a group of permutations,  but 
if $\Gamma_i\not=\Gamma$ then in general, the action is not  transitive any more. Indeed, the $\Gamma_i$-orbit of a component of ${\mathcal H}$ is a
$\Gamma_i$-invariant subset of ${\mathcal H}$, and these sets are transitively permuted 
by the action of the group $\Gamma_i\backslash \Gamma$; the latter being realized as the deck group of the covering $M_i\to M$, where $M_i:=\mathbb C\mathbb H^n/\Gamma_i$.
%, as can be seen by passing to the quotient manifold $M_i$ and the projection of $H_0$ to $M_i$, which covers the submanifold $H$ of $M$. 

\begin{proposition}\label{injectivity}
For the congruence tower $\Gamma\vartriangleright \Gamma_1\vartriangleright \Gamma_2\vartriangleright \cdots$, 
the distance in $\mathbb{C}\mathbb H^n$ between distinct hyperplanes in the 
same $\Gamma_i$-orbit for the action on $\mathcal{H}$ tends to infinity as $i\to \infty$. 
\end{proposition}
\begin{proof} The stabilizer of $H_0$ in 
$\Gamma_i$ is the congruence subgroup ${\rm Stab}(H_0)\cap \Gamma_i$ of ${\rm Stab}(H_0)$, in particular it is 
a normal subgroup of ${\rm Stab}(H_0)$. 
By Lemma \ref{malnormal}, the intersection 
${\rm Stab}(H_0)\cap \Gamma_i$ is a malnormal subgroup of 
$\Gamma_i$. 

The group $\Gamma$ acts transitively by conjugation on the conjugates 
of ${\rm Stab}(H_0)$ in $\Gamma$, and 
the orbit space for this action 
has a natural identification with ${\mathcal H}$. 
The group $\Gamma$ also acts transitively by conjugation on the 
conjugates of ${\rm Stab}(H_0)\cap \Gamma_i$. Since 
$\Gamma_i\cap {\rm Stab}(H_0)$ is a normal subgroup of 
${\rm Stab}(H_0)$, the stabilizer of ${\rm Stab}(H_0)\cap \Gamma_i$ for 
this action also equals 
${\rm Stab}(H_0)$. In other words, the orbit space for the action 
of $\Gamma$ on the conjugates of $\Gamma_i\cap {\rm Stab}(H_0)$ 
also can be viewed as the space of right cosets 
$\Gamma/{\rm Stab}(H_0)$.

For $\psi\in \Gamma$, the stabilizer of 
$\psi(H_0)$ in $\Gamma_i$ is 
$\Gamma_i\cap \psi {\rm Stab}(H_0)\psi^{-1}=\psi (\Gamma_i\cap {\rm Stab}(H_0))\psi^{-1}$.  
As a consequence, the space of $\Gamma_i$-orbits  for the 
action of $\Gamma_i$ on ${\mathcal H}$ can be identified
with the double coset space $\Gamma_i\backslash \Gamma/ {\rm Stab}(H_0)$.

The action of $\Gamma$ on $\mathbb{C}\mathbb{H}^n$ is proper and cocompact, and the
same holds true for the action of ${\rm Stab}(H_0)$ on $H_0$. Thus if we denote
by $d$ the distance in $\mathbb{C}\mathbb{H}^n$, then the following two properties hold. First, given any $H'\in \mathcal H$ distinct from $H_0$, the distance $d(H_0,H')$ is positive, achieved by a geodesic arc such that the starting point is contained in a fixed fundamental domain for the action of ${\rm Stab}(H_0)$ on $H_0$. Second, up to conjugation with elements of ${\rm Stab}(H_0)$, for any fixed number $R>0$  there are only finitely many 
$\phi\in \Gamma$ so that $d(\phi(H_0),H_0)\leq R$. Equivalently, there are only finitely
many components $H^\prime\in {\mathcal H}$ so that $d(H_0,H^\prime)\leq R$ and that
the starting point in $H_0$ of the shortest geodesic arc connecting 
$H_0$ to $H^\prime$ is contained in a fixed fundamental domain for the action of 
${\rm Stab}(H_0)$ on $H_0$. Any other $\hat H\in {\mathcal H}$ so that 
$d(H_0,\hat H)\leq R$ is contained in the ${\rm Stab}(H_0)$-orbit of these finitely many 
elements of ${\mathcal H}$.

Let $\phi_1,\dots,\phi_k$ be representatives of these finitely many 
${\rm Stab}(H_0)$-conjugacy classes, chosen as representatives 
of minimal translation length in their class (it will be convenient to set $\phi_0=\mathrm{Id}$). 
For each $i$ let $\pi_i:\Gamma\to \Gamma_i\backslash \Gamma$ be the quotient homomorphism. 
We have to show that 
\begin{equation}
\label{Pm}
\exists m\geq 1; \,\forall \ell=1,\dots, k, \quad \pi_m(\phi_\ell)\not\in \pi_m({\rm Stab}(H_0)).
\end{equation}
Indeed, what we want to show is that there exists $m\ge 1$ such that one has $d(H_0, \eta H_0)>R$ for any $\eta \in \Gamma_m$ not belonging to $\Gamma_m \cap \rm{Stab}(H_0)$.  By the definition of the $\phi_\ell$, this is equivalent to having $\eta (H_0) \neq \phi_\ell (H_0)$ for all $\ell=1,\dots, k$ and all $\eta\in \Gamma_m$. In order words, $\phi_\ell \notin \Gamma_m \cdot \mathrm{Stab}(H_0)$, i.e. for any $\gamma \in \mathrm{Stab}(H_0)$, the $\Gamma_m$-orbits $\Gamma_m \cdot \phi_\ell$ and $\Gamma_m \cdot \gamma$ are disjoint. But this is exactly saying that $\pi_m(\phi_\ell)\notin \pi_m({\rm Stab}(H_0)).$
%Namely, if \eqref{Pm} holds then $H_0$ and $\phi_\ell(H_0)$ are not contained in the same 
%$\Gamma_m$-orbit for the action on ${\mathcal H}$, and the same holds true for 
%$\psi \phi_\ell  (H_0)$ for any $\psi\in {\rm Stab}(H_0)$. Recall to this end that the action of 
%$\Gamma$ on ${\mathcal H}$ is the conjugation action on stabilizers.
%By the choice of the elements $\phi_\ell$
%and equivariance, 
%this implies that $d(\psi(H_0)),\eta \psi(H_0))\geq R$ for any $\eta\in \Gamma_m$ and any
%$\psi\in \Gamma$ which is what we wanted to show.
  
To see that \eqref{Pm} holds, we follow \cite{Be00} (Lemme principal). First of all, we observe that if $\pi_{m_0}(\phi_\ell)\notin \pi_{m_0}({\rm Stab}(H_0))$ for {\it some} $\ell$ and some $m_0\ge 1$, then it holds for that $\ell$ and for {\it any} $m\ge m_0$. Therefore it suffices to check that \eqref{Pm} holds for any given $\ell \in \{1, \ldots, k\}$. We fix such an $\ell$ until the end of the proof. Now, since (the real form of)
${\rm GL}(n,\mathbb{C})$ is not Zariski dense in (the real form of) 
${\rm GL}(n+1,\mathbb{C})$, for each $\ell$ 
there exists a polynomial $P$ on $GL(n+1,\mathbb{C})$ vanishing identically 
on ${\rm GL}(n,\mathbb{C})$ and such that $P(\phi_\ell)\not=0$. 
By the explicit form of $\Gamma$, we may assume that 
the coefficients of this polynomial are contained in the ring of integers $\mathcal O_L$ of the field $L$.
Namely, using the embedding of ${\rm GL}(n,\mathbb{C})$ into 
${\rm GL}(n+1,\mathbb{C})$ as the group of matrices $(a_{ij})$ in block form, with $a_{1j}=a_{j1}=0$ for $j>1$, 
a matrix in ${\rm GL}(n+1,\mathcal O_L)\setminus {\rm GL}(n,\mathcal O_L)$ has some nontrivial coefficients 
among the entries $\{a_{1j},a_{j1}\}$ and hence one can construct such a polynomial
explicitly.  
 
By the proof of Lemma 1 in \cite{Be00} or explicit considerations, there exists a prime ideal 
$\mathcal J_\ell\subset \mathcal O_L$  
such that for the quotient map $\eta:{\rm GL}(n+1,\mathcal O_L)\to 
{\rm GL}(n+1,\mathcal O_L/\mathcal J_\ell)$ it holds 
$\bar P(\eta(\phi_\ell))\not=0$ where $\bar P$ is the polynomial with coefficients
in the field 
$\mathcal O_L/\mathcal J_\ell$ obtained by applying the morphism $\mathcal O_L\to \mathcal O_L/\mathcal J_\ell$ to the 
coefficients of $P$. 
%Furthermore, we have $\bar P({\rm Stab}(H_0))=0$ since ${\rm Stab}(H_0)\subset {\rm GL}(n,\mathbb{C})$. 
Set $F:=\mathcal O_L/\mathcal J_\ell$ and $\Gamma^\prime={\rm ker}(\eta|_\Gamma)$. Then $\Gamma^\prime\vartriangleleft \Gamma$ has finite
index and by construction there exists some $m$ so that $\Gamma_m\vartriangleleft \Gamma^\prime$. Therefore we have the following commutative diagram 
\[
\begin{tikzcd}
& &\Gamma_m \backslash \Gamma \arrow[d, two heads]  \\
\mathcal O_L \arrow[d]&\Gamma \arrow[l, "P|_\Gamma", swap] \arrow[ur, "\pi_m"] \arrow[d,"\eta|_\Gamma", swap] \arrow[r,"\pi'"]& \Gamma'\backslash \Gamma  \arrow[dl, "\eta'"]\\
F&\mathrm{GL}(n+1, F) \arrow[l, "\bar P",]& 
\end{tikzcd}
\]
and $\eta|_\Gamma$ factors through $\pi'$ by definition of $\Gamma'$. Fix an arbitrary element $\gamma \in \mathrm{Stab}(H_0)$.  Since ${\rm Stab}(H_0)\subset {\rm GL}(n,\mathbb{C})$, we have $ P(\gamma)=0$. Therefore we also have $\bar P(\eta(\gamma))=0$. Moreover, $\bar P(\eta(\phi_\ell))\neq 0$ by construction of $P$ and $F$. This implies that $\eta(\gamma)\neq \eta(\phi_\ell)$, hence $\pi'(\gamma)\neq \pi'(\phi_\ell)$ since $\eta|_\Gamma$ factors through $\pi'$. In particular, one must have $\pi_m(\gamma)\neq \pi_m(\phi_\ell)$. Since $\gamma\in \mathrm{Stab}(H_0)$ is arbitrary, \eqref{Pm} follows and the proposition is proved. 
\end{proof}
%, and the projection of 
%$\phi_\ell$ to $\Gamma^\prime\backslash \Gamma$ 
%is not contained in the projection of ${\rm Stab}(H_0)$. 

% this shows the statement we
%are looking for for the element $\phi_\ell$. Repeating this construction for each of the 
%remaining elements completes the proof of the proposition.

Resuming the notations from Theorem \ref{mainapp}, we obtain.

\begin{corollary}\label{normalin}
For every $R>0$ there exists a number $j(R)>0$ so that for $j>j(R)$ 
the normal injectivity radius of 
a component of the preimage of $H$ in $M_j=\Gamma_j\backslash 
\mathbb{C}\mathbb{H}^n$ is at least $R$.
\end{corollary}
\begin{proof}
Let $H_j\subset M_j$ be a component of the preimage of $H$ 
and assume that
the normal injectivity radius of $H_j$ in $M_j$ is smaller than $\rho$ for some $\rho >0$. 
By definition, there is then a geodesic arc $\gamma\subset  M_j$ with endpoints on $H_j$ 
and of length at most $2\rho$ which is not homotopic with fixed endpoints into $H_j$. 
A lift of $\gamma$ to the universal covering is a geodesic of length at most $2\rho$
which connects two distinct lifts of $H_j$ 
to $\mathbb{C}\mathbb{H}^n$. 
As these lifts are components of ${\mathcal H}$ 
in the same orbit under the action of the fundamental group of 
$M_j$, it follows from Proposition \ref{injectivity} that $j\leq j(2\rho)$ for a number 
$j(2\rho)>0$ only depending on $\rho$.  
This is what we wanted to show.
\end{proof}

%\begin{remark}
%\label{further cover}
%Note that if $M'\to M_j$ is any finite étale cover with $j>j(R)$  (so that $M'=\mathbb C \mathbb H^n/\Gamma'$ with $\Gamma'< \Gamma_{j}$ of finite index), then the normal injectivity radius of a component of the preimage of $H$ in $M'$ is at least $j(R)$.
%\end{remark}

To complete the proof of Theorem \ref{mainapp},
we use the following main result of \cite{ST22} as the key ingredient. It is a concatenation of Corollary~2.15, Theorem~2.16 and Theorem~3.4 in {\it loc. cit.}
 
 \begin{theorem}[\cite{ST22}]\label{cupproduct}
Let $M$ be any congruence manifold and let $\eta\in H^{1,1}(M;\mathbb{C})$
be the Poincar\'e dual of a compact connected embedded totally geodesic hypersurface $H\subset M$.
Then 
\begin{enumerate}
\item There is a congruence cover $\Pi:M^\prime\to M$ so that $\Pi^*\eta$ is contained in the image of the cup product map $\wedge^2H^1(M^\prime;\mathbb{Q})\to H^2(M^\prime;\mathbb{Q})$.
\item For any integer $d\ge2$, there is a further finite étale cover $\Pi':M''\to M'$ such that $(\Pi \circ \Pi')^*\eta$ is divisible by $d$ in $H^2(M'', \mathbb Z)$. In particular, there is a finite cyclic cover $N\to M''$ of degree $d$ totally ramified over $(\Pi \circ \Pi')^{-1}(H)$.
\end{enumerate}
\end{theorem}

\begin{proof}[Proof of Theorem \ref{mainapp}]
Given a number $R>0$, Corollary \ref{normalin} shows that 
there exists a number $m_0>0$ so that for any $m\geq m_0$ the normal injectivity radius of any component of the preimage of $H$ under the cover $M_{m} \to M$ is at least $R$. 
 %Moreover, that property remains valid under any further étale cover $M'\to M_{m}$, see Remark~\ref{further cover}. 

Let $H^\prime$ be a component of the preimage of $H$ in $M_{m_0}$. By Theorem \ref{cupproduct}, by passing to a further congruence cover $\Pi_1:M_{m_1}\to M_{m_0}$, one can guarantee the existence of a finite ramified cover (of large, non-explicit degree) $N\to M_{m_1}$ which ramifies at order exactly $d$ along $\Pi_1^{-1}(H')$. This proves Theorem~\ref{mainapp}.
\end{proof}
%By Theorem \ref{cupproduct}, by passing to a further congruence cover
%$\Pi_1:M_{m_1}\to M_{m_0}$, we may
%assume that the Chern class $c_1(\Pi_1^*H^\prime)$ is contained in the 
%image of the cup product $\cup:\wedge^2H^1(M_{m_1};\mathbb{Q})\to H^2(M_{m_1};\mathbb{Q})$. 
%For a given number $p>0$ we then can pass to a further congruence cover
%$\Pi_2:M_{m_2}\to M_{m_1}$ so that the pull-back of $c_1(H^\prime)$ to $M_{m_2}$ is divisible by $p$. 
%Following Proposition 5.1 of \cite{ST22} and its proof, this implies that 
%there exists a degree $p$ cover of $M_\ell$, branched along 
%$(\Pi_j\circ \Pi_\ell)^{-1}(H^\prime)$. This completes the proof of the theorem.

The manifolds from our main theorem are the branched coverings 
constructed in Theorem \ref{mainapp}. That these manifolds are not ball quotients 
was established in \cite{ST22} (see the proof of Theorem 1.5 of \cite{ST22}). 

\begin{theorem}[Theorem 1.5 of \cite{ST22}]
A covering of a compact ball quotient, branched along a smooth 
embedded totally geodesic submanifold, is not a quotient of the ball.  
\end{theorem}

\section{Analysis of the Kähler--Einstein cone metric}\label{sec:analysis}

Let $M=\Gamma\backslash B$ be a compact ball quotient of complex dimension $n$ 
where $\Gamma$ is a torsion free arithmetic lattice in ${\rm PU}(n,1)$ of simple type.
We assume that $M$ contains 
a smooth totally geodesic embedded subvariety $D\subset M$ of codimension one. 

\subsection{Ampleness of the adjoint divisor}

The following observation is well known, see for example \cite{ST22}. As we shall use some 
more specific information, we provide the proof. 

\begin{lemma}
\label{normal bundle}
The normal bundle $N_D=\mathcal O_M(D)|_D$ satisfies 
\[c_1(N_D) = -\frac 1n c_1(K_D).\]
\end{lemma}

\begin{proof}
On the ball $B$, the complex hyperbolic metric $-\frac{i}{4}\partial \bar \partial \log(1-|z|^2)$ descends to a Kähler metric $\omega_B$ on $M$ which satisfies $\mathrm{Ric} \, \omega_B=-2(n+1)\,\omega_B$. In particular, we have 
\[c_1(K_M)=2(n+1)\,[\omega] \quad \mbox{in} \,\,H^2(M, \mathbb R).\] Now, one can assume without loss of generality that a connected component of the inverse image of $D$ in the universal cover $B$ of $M$ is given by the equation $z_1=0$. Performing the same computation as before, one sees that $\mathrm{Ric} \, {\omega_B}|_D=-n\, {\omega_B}|_D$. In particular, we have \begin{equation}
  \label{canonical}  
  c_1(K_M)|_D=\frac{n+1}{n}\, c_1(K_D).
\end{equation} 
Combining \eqref{canonical} with the adjunction formula $(K_M+D)|_D\simeq K_D$, we get
\begin{equation}
  \label{canonical 2}  
  c_1(D)|_D=-\frac{1}{n}\, c_1(K_D).
\end{equation} 
which proves the lemma. 
\end{proof}

\begin{remark}
    It follows from the lemma above that the normal bundle of $D$ is negative. By Satz 8 in \textsection 3 of \cite{Gr62},
   it is possible to find a surjective holomorphic map $\pi:M\to M^*$ where $M^*$ is a compact normal analytic space, $\pi$ contracts $D$ to a point and $\pi$ is an isomorphism when restricted to the complement of $D$.  Then the singularities of $M^*$ are never log canonical. Indeed, assuming that $K_{M^*}$ is $\mathbb Q$-Cartier, then we would have $K_M=\pi^*K_{M^*}+bD$ for some $b\in \mathbb Q$. Restricting the formula to $D$ and using \eqref{canonical}-\eqref{canonical 2}, we infer that $b=-(n+1)<-1$.
\end{remark}

\begin{lemma}
\label{ampleness}
Let $a\in [0,\infty)$. The $\mathbb R$-line bundle $K_M+aD$ is ample if and only if $a<n+1$.
\end{lemma}

\begin{proof}
 Let us first observe that $K_M+aD$ is big for any $a\ge 0$, as a sum of an ample divisor and an effective divisor. Moreover, the non-Kähler locus (or augmented base locus) of $K_M+aD$ is clearly included in $D$.
 
 Next, \eqref{canonical} and \eqref{canonical 2} yield
 \[(K_M+aD)|_D \equiv \frac{n+1-a}{n}\, K_D\]
 and the latter is ample if and only if $a<n+1$. In particular, the same holds for the restriction of $K_M+aD$ to any irreducible component of its non-Kähler locus. The conclusion of the lemma now follows from Theorem 3.17 (iii) in \cite{B04}.
\end{proof}

\subsection{Kähler--Einstein cone metrics}
\label{ssec KE cone}

Given a real number $a\in (0,1)$, one says that a Kähler metric $\omega$ on $M\setminus D$ 
has cone singularities with cone angle $2\pi (1-a)$ along $D$ if given any local holomorphic 
system of coordinates $(z_1, \ldots, z_n)$ on an open set $U\subset M$ such that $D\cap U=\{z_1=0\}$, the Kähler metric $\omega|_{U\setminus D}$ is quasi-isometric to the model metric
\[\omega_a:=\frac{idz_1\wedge d\bar z_1}{|z_1|^{2a}}+\sum_{j=2}^n idz_j\wedge d\bar z_j.\]
That is, there exists $C>0$ such that we have 
\[C^{-1} \omega_a \le \omega \le C \omega_a \quad \mbox{on } \,\, U\setminus D.\]

Let us now fix $a\in (0,1)$. Since $K_M+aD$ is ample by Lemma~\ref{ampleness}, it follows from \cite{GP} (see also \cite{Br,JMR}) that there exists a unique Kähler metric $\omega_{{\rm KE}, a}$ on $M\setminus D$ such that 
\begin{enumerate}
    \item[$\bullet$] $\mathrm{Ric} \, \omega_{{\rm KE}, a}= -2(n+1)\, \omega_{{\rm KE}, a}$,
    \item[$\bullet$] $\omega_{{\rm KE}, a}$ has cone singularities with cone angle $2\pi (1-a)$ along $D$.
\end{enumerate}
Moreover, when $a$ is of the form $a=1-\frac 1m$ for some integer $m \ge 2$, then it is well-known that $\omega_{{\rm KE}, a}$ is an orbifold Kähler metric, see e.g. \cite{Faulk}. What we mean is the following. Given a local chart $U\simeq \Delta^n$ as before, consider the branched cover $p:\Delta^n\to U$ given by $p(z_1, \ldots, z_n)=(z_1^m, z_2, \ldots, z_n)$. Then $p^*(\omega_{{\rm KE}, a}|_{U\setminus D})$ extends to a smooth Kähler metric on the whole $\Delta^n$.

\subsection{Cut-off functions}

Let us first work on the ball $B\subset \mathbb C^n$ endowed with its Bergman metric $\omega_B$ and consider the lower dimensional 
ball $B_0:=B\cap \{z_1=0\}$. We can normalize $\omega_B$ such that 
\[\omega_B=-\frac i 4 \partial \bar \partial \log (1-|z|^2)=\frac{1}{4(1-|z|^2)^2}\sum_{1\le j,k\le n}\big((1-|z|^2)\delta_{jk}+\bar z_j z_k\big) idz_j\wedge d\bar z_k.\]
With this normalization, $\omega_B$ has constant holomorphic holomorphic sectional (resp. bisectional) curvature $-4$ (resp. 
$-2$), 
its sectional curvatures lie in $[-4, - 1]$ and we have 
\[\mathrm{Ric} \, \omega_B:=-\frac{i}{2}\partial \bar \partial \log \det \omega_B^n=-2(n+1)\omega_B.\]

Given a point $p\in B$ with coordinates $(z_1, \ldots, z_n)$, we let $d(p,B_0)$ be the geodesic distance to $B_0$ with respect to $g_B$. 

\noindent
{\bf Fact:} For $p=(z_1,w)\in B$, we have 
$\cosh^2 (d(p,B_0))= \frac{1-\vert w\vert^2}{1-\vert w\vert^2-\vert z_1\vert^2}$.

For a justification, observe that the function $p\mapsto d(p, B_0)$ is invariant under $S^1\times \mathrm{PU}(n-1, 1)$.
Hence there is no loss of generality assuming that $p=(x_1,x_2,0, \ldots, 0)$ with $x_i\in \mathbb [0,1)$,
$x_1=\vert z_1\vert$ and $x_2=\vert w\vert$. 
Since $\{z_i=0 \text{ for }i\geq 3\}$ is totally geodesic in $B$, and the same holds true for the 
totally real plane $\{(r,s)\mid r,s\in \mathbb{R}, r^2+s^2<1\}\subset B\cap \mathbb{C}^2$, 
it then suffices to compute the distance between $(x_1,x_2)\in B\cap \mathbb{C}^2$ and 
a point $(0,r)\in B\cap \mathbb{C}^2$ with $r\in [0,1)$. 
The computation on p.15 of \cite{P03} shows that 
\[\cosh^2\Big(d(p,(0,r))\Big)=\frac{(1-rx_2)^2}{(1-x_1^2-x_2^2)(1-r^2)}\]
which is minimized at $r=x_2$. Note that \cite{P03} uses a curvature normalization which is different from ours.
This is the desired formula.

Set $u:=\frac{1-\vert w\vert^2}{1-\vert w\vert^2-\vert z_1\vert^2}=1+\frac{|z_1|^2}{1-|z|^2}$ so that 
$d(\cdot, B_0)= \log(\sqrt{u}+\sqrt{u-1})$ satisfies $\frac 12 \log u \le d(\cdot, B_0) \le \frac 12 \log(4u)$. It is easy to check that $\log u$ is smooth and it has uniformly bounded covariant derivatives at any order. 
Next, let $\xi : \mathbb R\to [0,1]$ be a smooth, non-increasing function such that $\xi \equiv 1$ on $(-\infty,1]$ and $\xi \equiv 0$ on $[\frac 32,+\infty)$.  For any $R>1$, we set 

\begin{equation}\label{chi R}
\widetilde \chi_R:=\xi\Big(\frac{2\log u}{R}\Big)=  \xi\left(\frac{2}{R} \log\left(1+\frac{|z_1|^2}{1-|z|^2}\right)\right).
\end{equation}
The latter function is smooth and satisfies $\widetilde \chi_R \equiv 1$ on 
$\big\{d(\cdot, B_0)\le \frac R4\big\}$ and $ \widetilde \chi_R \equiv 0$ on $\big\{d(\cdot, B_0)\ge \frac R2 \big\}$ as long as $R\ge 6$, thanks to the above fact. Moreover,  by the chain rule we see that for any integer $k\ge 0$, there exists a universal constant $C=C(k)>0$ independent of $R$ satisfying 

\begin{equation}
\label{cut off}
|\nabla^k \widetilde \chi_R |_{\omega_B} \le \frac{C(k)}{R^k}.
\end{equation}
Finally, we see that by construction, the function $\widetilde \chi_R$ on $B$ is invariant under $S^1\times \mathrm{PU}(n-1, 1)$. \\

Let us now go back to our compact ball quotient $M= \Gamma \backslash B$ with its embedded totally geodesic smooth 
connected hypersurface $D\subset M$. From now on, we fix a number $d \ge 2$. By Theorem~\ref{mainapp} (see also Remark~\ref{reformulation2}), there exists a tower of finite covers
\[
\begin{tikzcd}
&  M_{d, R} \arrow[d, "p_{d,R}"]\\
M & M_R \arrow[l, "\Pi_R"]
\end{tikzcd}
\] 
where $\Pi_R$ is étale, $p_{d,R}$ ramifies at order $d$ along a fixed connected component $D_R$ of $\Pi_R^{-1}(D)$ and is étale elsewhere and $D_R\subset M_R$ has normal injectivity radius at least $R$.  In terms of canonical bundles, we have
\[K_{M_{d,R}}=p_{d,R}^*\left(K_{M_R}+\big(1-\frac 1d\big)D_R\right).\]

%and consider the \'etale cover $M'\to M$ such that $M'$ admits a further ramified cover $p: M_d\to M'$ such that $M_d$ is smooth and $p$ is branched at order $d$ along $D$ and étale elsewhere. The existence of such objects is due to Stover and Toledo and was recalled in Section~\ref{sec ST}. In order to lighten notation, we replace from now on $M$ with $M'$; that is, we assume that $M$ admits a finite cover $p: M_d\to M$ branched exactly at order $d$ along $D$. 
% Observe that the fiber product $M_{d,R}:=M_d\times_M M_R$ 

\begin{remark}
It is important to keep in mind that as $R$ grows, we have no control on the growth of $\deg(p_{d,R})$ so that $p^{-1}_{d,R}(D_R)$ will have a large number of connected components. In what follows, we will perform the analysis directly on $M_R$ and only rely on the existence of $M_{d,R}$ in a qualitative way to desingularize the Kähler--Einstein metric associated to the pair $(M_R, (1-\frac{1}{d})D_R)$. 
\end{remark}
%\begin{remark}
%Since we have no control on $\deg(\Pi_R)$ in terms of $R$, we have chosen to  pick out a single component 
%$D_R$ of $\Pi_R^{-1}(D)$ so that only one copy of the model metric $\omega_d$ 
%needs to be glued to the complex hyperbolic metric. 
%%It is conceivable that one could similarly glue $\deg(\Pi_R)$ copies of $\omega_d$ 
%%as $M_R$ can be constructed in such a way that the collars of size $R$ of the component 
%%of $\Pi_R^{-1}(D)$ are disjoint, but the construction may get technically more involved.  
%\end{remark}

Let us now reformulate the defining property of $M_R$. We fix one connected component $V$ of the 
preimage of the connected smooth divisor 
$D_R\simeq D$ in the universal cover $B$ of $M_R$ and we let $\Gamma_0:=\mathrm{Stab}_\Gamma(V)$ be the stabilizer of $V$ in 
$\Gamma_R$. We have $\Gamma_0\backslash V = D$. Without loss of generality, one can assume that 
$V=B_0=(z_1=0)$. Since the collar size of $D_R$ in $M_R$ is at least $R$, the tubular neighborhood of 
radius $R$ about $D_R$ equals the projection of the tubular neighborhood of radius $R$ about $B_0$ by the action of 
$\Gamma_0$. Thus 
there is a holomorphic, isometric embedding 
\[j_R:\big\{x\in M_R; d(x, D_R) < R\big\} \longrightarrow \Gamma_0 \backslash \big\{p\in B, d(p, B_0) < R\big\}\]
with respect to the complex hyperbolic metric. As explained above, the cut-off function $\widetilde \chi_R$ defined in \eqref{chi R} is invariant under the stabilizer of $B_0$ hence it makes sense to define
\[\chi_R:=j_R^*\widetilde \chi_R.\]

\subsection{The glued metric}
Recall from Section~\ref{sec model KE} that the domain \[\Omega_d:=\{|z_1|^{2d}+\sum_{i=2}^n |z_i|^2<1\}\] has a Kähler--Einstein metric with Ricci constant $-2(n+1)$  which is invariant under $G=\mathrm{Aut}(\Omega_d)$ and therefore descends to a Kähler--Einstein metric on $B$ with {\em cone singularities} of angle $2\pi(1-\frac 1d)$ along $B_0$, invariant under $\Gamma_0$. We denote by $\omega_d$ the induced metric on $\Gamma_0\backslash B$; by abuse of notation we will also denote by $\omega_d$ its pull back to $\big\{x\in M_R; d(x, D_R) < R\big\}$ via $j_R$. 

Let us define $U_R:=\big\{x\in M_R; d(x, D_R) < R\big\} \subset M_R$ on which the function $\chi_R$ is well-defined and compactly supported. We introduce the smooth function $F=\frac{1}{2(n+1)}\log \frac{\omega_B^n}{\omega_d^n}$ on $U_R\setminus D_R$ and set 
\[\omega_R:=\omega_d+dd^c\Big((1-\chi_R) F\Big)=\omega_B-dd^c(\chi_R F).\]
This current is a priori only defined on $U_R$. However, since $\omega_B$ and $\omega_d$ are both Kähler--Einstein metrics 
with the same Einstein constant on $U_R\setminus D_R$, an elementary computation shows that $\omega_R$ coincides with $\omega_B$ on $U_R\setminus U_{\frac R2}\subset \{\chi_R=0\}$ and hence we can extend $\omega_R$ to the whole $M_R$ by setting $\omega_R:=\omega_B$ on $M_R\setminus U_R$.  It is not difficult to show that $\omega_R\in \frac{1}{2(n+1)}c_1(K_{M_R}+(1-\frac 1d) D_R)$. Moreover, $\omega_R$ coincides with $\omega_d$ on $U_{\frac{R}4}$. 
It remains to analyze the behavior of $\omega_R$ on the gluing region $U_{\frac R2}\setminus U_{\frac R4}$. 

In what follows, we will denote by $C(k)$ a constant that depends on a given integer $k\in \mathbb N$ (and implictly on $n$ and $d$) but not on the parameter $R$. The actual value of $C(k)$ may change from line to line but it subject to the constraints recalled above. From Theorem~\ref{thm:comparison}, there exists for any integer $k\ge 0$ a positive number $a=a(d,k)$ such that we have the following decay
\[|\nabla^k(\omega_d-\omega_B)|_{\omega_B} \le C(k)e^{-a R} \quad \mbox{on} \, \, U_{\frac R2}\setminus U_{\frac R4}.\] 
 In particular, the covariant derivatives of $F$ decay in $O(e^{-aR})$. Since the covariant derivatives of $\chi_R$ are bounded (actually they decay polynomially in $R$, cf \eqref{cut off}) it follows that 
\[|\nabla^k (dd^c((1-\chi_R) F)) |_{\omega_B}\le C(k)e^{-a R}\quad \mbox{on} \, \, U_{\frac R2}\setminus U_{\frac R4}.\] Putting everything together, one obtains the following identity
\begin{equation}
\label{glued asymp}
|\nabla^k(\omega_R-\omega_B)|_{\omega_B} \le C(k)e^{-aR} \quad \mbox{on} \, \, M_R\setminus U_{\frac R4}.\
\end{equation}
In particular, it follows from the third item of Theorem~\ref{negativeonreal} that for $R$ large enough, the sectional curvature of 
$\omega_R$ bounded above by a negative independent of $R$. More precisely, \eqref{glued asymp} and \eqref{sup K} imply that for $R$ large enough (depending on $d$), we have
\begin{equation}
\label{sup K R}
    \sup_{x\in M_R\setminus D_R} \sup_{\substack{P\subset T_xM_R \\ \mathrm{plane}}}K_{g_R}(P)= -(n+1)+\frac{n}{f_d(0)^2}
\end{equation}
since the sectional curvatures of $\omega_B$ lie in $[-4,-1]$ and $-(n+1)+\frac{n}{f_d(0)^2}\in (-1, 0)$

Let us now analyze the Ricci potential of $\omega_R$. From the definition of $\omega_R$, it is straightforward to deduce that 
\[\Ric \omega_R+2(n+1)\omega_R=2(n+1) dd^c h_R+\big(1-\frac 1d)[D_R]\]
on $M_R$ where
\[h_R:=-\frac{1}{2(n+1)}\log \frac{\omega_R^n}{\omega_B^n}- \chi_R F\]
 is smooth function on $M_R$ satisfying 
 \[h_R\equiv 0 \quad \mbox{on} \quad (M_R\setminus U_{\frac R2}) \cup U_{\frac R4}\] as well as 
\begin{equation}
\label{hR decay}
|\nabla^kh_R|_{\omega_R} \le C(k)e^{-aR} \quad \mbox{on} \, \,  U_{\frac R2}\setminus U_{\frac R4}.\
\end{equation}
In particular, we have 
\[|\Ric \omr+2(n+1)\omega_R|_{\omega_R}= O(e^{-aR}).\]

\subsection{Curvature of the Kähler--Einstein cone metric on $M_R$}

Thanks to Section~\ref{ssec KE cone}, there exists a unique Kähler--Einstein metric $\womr$ on $M_R$ with cone angle $2\pi(1-\frac 1d)$ along $D_R$ and Einstein constant $-2(n+1)$. Note that since we only picked one component $D_R$ of $\Pi_R^{-1}(D)$, the metric  $\womr$ is \emph{not} the pullback by $\Pi_R$ of the Kähler--Einstein metric for the pair $(M, (1-\frac 1d)D)$. The forms $\omr$ and $\womr$ are orbifold Kähler metrics, that is, they are genuine Kähler metrics on $M_{R}\setminus D_R$, and their pullbacks by $\Phi_d: \Omega_d\to B$ (after first pulling back to the universal cover $\widetilde{U}_R\subset B$) is smooth. Equivalently, both pullbacks 
\begin{equation}
    \label{pull back metrics}
    \widehat \omega_{d,R}:=p_{d,R}^*\womr \quad \mbox{and} \quad \omega_{d,R}:=p_{d,R}^*\omr
\end{equation}
are genuine Kähler metrics on $M_{d,R}$. Since $\omr$ and $\womr$ both belong to the cohomology class
$\frac{1}{2(n+1)}c_1(K_{M_R}+\big(1-\frac 1d\big) D_R)$, one can uniquely write $\womr = \omr+dd^c \phi_R$ where $\phi_R$ solves the Monge-Ampère equation 
\begin{equation}
    \label{MAR}
    (\omr+dd^c \phi_R)^n=e^{2(n+1)(\phi_R+h_R)}\omega_R^n.
\end{equation}

Let us now derive some uniform estimates (as $R$ varies) on $\womr$ and $\phi_R$. First, since the holomorphic bisectional curvature of $\omr$ is bounded above by a negative constant independent of $R$, Theorem 2 of \cite{Y78} shows that 
\[\womr \ge C^{-1} \omr.\]
Here and in what follows, $C$ is a positive constant independent of $R$ which may vary from line to line. Next, since $\omr$ and $\womr$ have Ricci curvature bounded below (say by $-2(n+2)$) we can apply Theorem 3 of \cite{Y78} to conclude that the volume elements of both metrics are uniformly comparable. Given the above estimate, this implies that one has an estimate of the form
\begin{equation}
    \label{C2 est}
    C\omr \ge \womr \ge C^{-1} \omr.
\end{equation}

Consider the orbifold smooth function $\phi_R+h_R$ from the identity (\ref{MAR}). 
At a point $x_R$ where it attains its maximum, its Hessian is nonpositive hence 
$\womr(x_R) \le \omr(x_R)-dd^c h_R(x_R)$. To be precise, one works in the branched cover $M_{d,R}$ where the objects become smooth,
and then one can descend the estimates which \emph{do not depend} on the cover $p_{d,R}$. Since $|dd^c h_R|_{\omr}=O(e^{-aR})$, we infer from the Monge-Ampère equation satisfied by $\phi_R$ that $(\phi_R+h_R)(x_R) \le Ce^{- aR}$ hence the same holds on the whole $M_R$. One can similarly use the minimum principle to see that $\phi_R+h_R\ge -Ce^{-a R}$. By \eqref{hR decay}, we obtain
\begin{equation}
    \label{phiR decay}
    \sup_{M_R} |\phi_R| \le C e^{-a R}.
\end{equation}

The remaining task is to improve this $C^0$ decay to order four decay on $\phi_R$ which will guarantee that the curvature of $\womr$ is close to that of $\omr$, hence it is negative too. For $k\in \mathbb N$, and $f$ a smooth orbifold function on $M_R$, we set $\|f\|_{C^k(M_R)}:=\sup_{M_R} \sum_{j=0}^k|\nabla^k f|_{\omega_R}$. We will show that $\|\phi_R\|_{C^5(M_R)}$ gets arbitrarily small if $R$ is chosen large enough. It is convenient to assume that $R$ is integer valued. Let $x_R\in M_R$ be such that $\|\phi_R\|_{C^5(M_R)}= \sum_{j=0}^5|\nabla^5 \phi_R(x_R)|_{\omega_R}$. Up to extracting subsequences, we only have to consider the following two possibilities. 

\emph{Case 1.} $\limsup_{R\to +\infty} d(x_R, D_R)<+\infty$.

Let us choose a constant $L>0$ such that $d_{\omega_R}(x_R, D_R)\le L$. Using $j_R$, one can embed $\{d_{\omega_R}(\cdot; D_R)\le 3L\}$ in $\Gamma_0\backslash B$ for $R$ large enough. Let $\sigma_R$ be the composition $\Omega_d\overset{\Phi_d}{\to} B\to \Gamma_0\backslash B$. It satisfies $\sigma_R^*\omega_R= \omega_d$. Given the structure of the automorphism group $G$ of the pair $(\Omega_d,(z_1=0))$ one can find a point $p_R\in \Omega_d$  such that $d_{\omega_d}(p_R, 0)\le L$ and  $\sigma_R(p_R)=x_R$. 

From now on, we work on $B_{\omega_d}(0, 3L)\subset \Omega_d$ and define $\widetilde \phi_R:=\sigma_R^*\phi_R$. We can pull back the Monge-Ampère equation \eqref{MAR} there. Since we have the Laplacian estimate \eqref{C2 est}, one can appeal to Evans-Krylov theorem and Schauder estimates to get uniform estimates for the $C^6$ norm of $ \widetilde \phi_R$ on $B_{\omega_d}(0, 2L)$ with respect to $\omega_d$. In particular, up to extracting again, we can assume that $\widetilde \phi_R$ converges in $C^5$ on a slightly smaller ball as $R\to +\infty$. By uniqueness of the limit, we see from \eqref{phiR decay} that $\widetilde \phi_R$ converges to zero in $C^5$ on that set. Given the choice of $x_R$ and since $p_R\in \bar B_{\omega_d}(0,L)$ it follows that
\[\|\phi_R\|_{C^5(M_R)}=\sum_{j=0}^5|\nabla^j \widetilde \phi_R(p_R)|_{\omega_d} \underset{R\to +\infty}{\longrightarrow}0.\]

\emph{Case 2.} $\liminf_{R\to +\infty}   d(x_R, D_R)=+\infty$.

For every integer $k\ge 0$, we have 
\[\sup_{B_{\omega_R}(x_R, 1)} |\nabla^k(\omega_R-\omega_B)|_{\omega_B} \underset{R\to+\infty}\longrightarrow 0\]
thanks to \eqref{glued asymp} and the fourth item in Theorem~\ref{negativeonreal}. Now we pull back our objects to the universal cover $\pi_R : B\to M_R$. Let $p_R\in B$ such that $\pi_R(p_R)=x_R$. By transitivity of the automorphism group of $(B, \omega_B)$, we can find $\mu_R\in \mathrm{Aut}(B, \omega_B)$ such that $\mu_R(0)=p_R$. Let us now consider $\sigma_R:=\pi_R\circ \mu_R$ and $\widetilde \phi_R:=\sigma_R^*(\phi_R|_{B(x_R, 1)})$. We have $\sigma_R(0)=x_R$ and $\sup_{B_{\omega_B}(0,1)}|\nabla^k (\sigma_R^*\omega_R-\omega_B)|\to 0$ for any integer $k\ge0$.  Similarly to the previous step, we can pull back the Monge-Ampère equation \eqref{MAR} by $\sigma_R$. Since we have the Laplacian estimate \eqref{C2 est}, one can appeal to Evans-Krylov theorem and Schauder estimates to get uniform $C^6$ estimates for $\widetilde \phi_R$ on $B_{\omega_B}(0, \frac 34)$ with respect to $\omega_B$. Up to extracting again, we can assume that $\widetilde \phi_R$ converges in $C^5$ on $B_{\omega_B}(0, \frac 12)$ as $R\to+ \infty$. By uniqueness of the limit, we see from \eqref{phiR decay} that $\widetilde \phi_R$ converges to zero in $C^5$ on that set. It follows that 
\[\|\phi_R\|_{C^5(M_R)}\le 2\sum_{j=0}^5|\nabla^j \widetilde \phi_R(0)|_{\omega_B} \underset{R\to +\infty}{\longrightarrow}0.\]

In conclusion, we have showed that 
\[\limsup_{R\to +\infty}\|\phi_R\|_{C^5(M_R)}=0,\] 
hence
\begin{equation}
    \label{smooth convergence}
    \lim_{R\to +\infty} \sup_{M_{d,R}}\sum_{j=0}^3|\nabla^j(\widehat \omega_{d,R}-\omega_{d,R})|_{\omega_{p,R}}=0,
\end{equation}
where $\widehat \omega_{d,R}$ and $\omega_{d,R}$ are defined in \eqref{pull back metrics}.

\begin{proof}[Proof of the main theorem]
We can now complete the proof of the theorem announced in the introduction. 

The forms $\widehat \omega_{d,R}$ and $\omega_{d,R}$ are genuine Kähler metrics on $M_{d,R}$ which are asymptotically close in the sense of \eqref{smooth convergence} as $R\to +\infty$. Since the sectional 
curvature of the Kähler metric $\omega_{d,R}$ on $M_{d,R}$ belongs to some interval $[-b^2, -a^2]$ for some numbers $0<a<b$ independent of $R$ by Theorem~\ref{negativeonreal}, it follows that the sectional curvature of the Kähler--Einstein metric $\widehat \omega_{d,R}$ satisfies the same property as long as $R$ is chosen large enough. This proves the theorem.
\end{proof}

\emph{Infinite family of examples.} 
One can say more, as claimed in the lines below the theorem in the introduction. 

Set $k_d:=(n+1)-\frac{n}{f_d(0)^2}$ 
and $\varepsilon_d:=\frac 12(k_d-k_{d+1})$
which is positive and goes to $0$ as $d\to +\infty$. Given \eqref{sup K R}, one can for any fixed $d$ choose $R=R(d, \varepsilon_d)$ large enough so that 
\[\Big|\sup_{M_{d,R}} K_{\widehat g_{d,R}} - k_d\Big| \le \varepsilon_d.\]
It follows that the quantity
\[\sup_{M_{d,R_d}} K_{\widehat g_{d,R_d}}\]
is strictly increasing with $d$. In particular, given two integers $d,d'\ge 2$,  the universal covers of $(M_{d,R_d}, \widehat\omega_{d,R_d})$ and $(M_{d',R_{d'}}, \widehat\omega_{d',R_{d'}})$ are not isometric unless $d=d'$. By uniqueness of the complete Kähler--Einstein metric on $\widetilde M_{d,R_d}$, this implies 
that $\widetilde M_{d,R_d}$ and $\widetilde M_{d',R_{d'}}$  are not biholomorphic when $d\neq d'$.

\emph{Very strong negativity.} 
Let us recall the notion of very strong negativity introduced by Siu \cite{Siu}. 
Let $(M,\omega)$ be a Kähler manifold written locally $\omega= \frac{i}{2}\sum_{i,j}g_{i\bar j}dz_i\wedge d\bar z_j$. The curvature tensor is given by $R_{i\bar j k\bar \ell}=-g_{i\bar j, k\bar \ell}+g^{s\bar t}g_{s\bar j, k}g_{i\bar t, \bar \ell}$. We say that the curvature tensor of $(M,\omega)$ is very strongly negative if 
\[\sum_{i,j,k,\ell} R_{i \bar j k \bar \ell}\xi^{i\bar j} \overline{\xi^{\ell \bar k}}\]
is negative for arbitrary complex numbers $\xi^{i\bar j}$ such that $\xi^{i\bar j}\neq 0$ for at least one pair of indices $(i,j)$. If $M$ is compact, it is equivalent to the existence of $c>0$ such that $\sum_{i,j,k,\ell} R_{i \bar j k \bar \ell}\xi^{i\bar j} \overline{\xi^{\ell \bar k}} \le -c |\xi|_\omega^2$ for any local holomorphic section $\xi$ of $T_M\otimes T_M$.

Because of the twist of indices in the above negativity condition, the curvature tensor of $(M,\omega)$ is negative if and only if the holomorphic cotangent bundle $\Omega_M$ equipped with the hermitian metric induced by $\omega$ is Nakano positive. Using an other terminology, it can be rephrased by saying that the holomorphic tangent bundle $T_M$ is dual Nakano negative with respect to the hermitian metric induced by $\omega$.

Let $\alpha=\frac{i}{2}\sum_{i,j}h_{i\bar j}dz_i\wedge d\bar z_j$ be a real $(1,1)$-form and let $H_{i\bar j k \bar \ell}=-(h_{i\bar j}h_{k\bar \ell}+h_{i\bar \ell}h_{k \bar j})$ be the $(0,4)$ tensor induced by $\alpha$ (or $h$).  If $\alpha$ is positive (resp. semipositive), then $H$ is very strongly negative (resp. strongly seminegative). 
Indeed, one can assume that $h_{i\bar j}= \lambda_i \delta_{i\bar j}$ for some $\lambda_i > 0$ (resp. $\lambda_i \ge 0)$ and then $-\sum_{i,j,k,\ell}H_{i\bar j k \bar \ell}\xi^{i\bar j} \overline{\xi^{\ell \bar k}}= |\sum_i \lambda_i \xi^{i\bar i}|^2+\sum_{i,j} \lambda_i\lambda_j |\xi^{i\bar j}|^2.$
This applies to the curvature tensor of the ball endowed with the Bergman metric and shows that the latter has very strongly negative curvature tensor. Similarly, if $f$ is a real function, then the tensor $-f_if_{\bar j}f_k f_{\bar \ell}$ is very strongly seminegative since $f_if_{\bar j}f_k f_{\bar \ell}\xi^{i\bar j} \overline{\xi^{\ell \bar k}}= |\sum_{i,j} f_if_{\bar j}\xi^{i\bar j}|^2$.

This discussion applies to the curvature of the Kähler-Einstein metric $\omega_\alpha$ on $\Omega_\alpha$ as it was showed in Theorem 2 of \cite{Bl86} that its curvature tensor $R_{i\bar j k \bar \ell}$ can be decomposed as a sum of terms
\[R_{i\bar j k\bar \ell}=-A (g_{i\bar j}g_{k\bar \ell}+g_{i\bar \ell}g_{k \bar j})-B(\psi_{i\bar j}\psi_{k\bar \ell}+\psi_{i\bar \ell}\psi_{k\bar j})-C\tau_i\tau_{\bar j}\tau_{k}\tau_{\bar \ell}\]
where $A,B,C$ are semipositive functions such that $A\ge \frac{2}{n\alpha+1}$, $\psi:=\log |z_1|^2-\frac{1}{\alpha}\log(1-|z'|^2)$ is plurisubharmonic and $\tau=e^\psi$.

Since $\omega_B$ and $\omega_d$ have very strongly negative curvature tensor, it follows from \eqref{smooth convergence} that the Kähler-Einstein metric $\widehat \omega_{d,R}$ shares the same property as long as $R$ is chosen large enough.

\end{document}